\documentclass[letterpaper,reqno,11pt]{amsart}

\usepackage{preamble}
\usepackage{mathtools}

\title{Linkage axioms for generic tropical oriented matroids}
\author{Yuan Yao}

\begin{document}

\begin{abstract}
    We present a comprehensive overview of Ardila and Develin's (generic) tropical oriented matroids, as well as many related objects and their axiomatics. Moreover, we use a unifying framework that elucidates the connections between these objects, including several new cryptomorphisms and properties.
\end{abstract}

\maketitle

\section{Introduction}\label{sec:intro}

\emph{Tropical oriented matroids} are a combinatorial abstraction of \emph{tropical hyperplane arrangements}, first defined by Ardila and Develin (\cite{TropicalOM}). Over the past few decades, these objects have been shown to be in bijection with several other geometric objects:

\begin{itemize}
    \item Subdivisions of product of two simplices (\cite{TropicalOM, TriangulationTOM, TropoRep}),
    \item Mixed subdivisions of a dilated simplex (via the \emph{Cayley trick}, see \cite{CayleyZono, CayleyTri}), and
    \item Tropical pseudo-hyperplane arrangements (\cite{TropoRep}).
\end{itemize}

In the \emph{generic} case, where the corresponding tropical (pseudo-)hyperplanes intersect generically, these matroids biject to \emph{triangulations} of products of two simplices and \emph{fine} mixed subdivisions of dilated simplices, the former of which have been of interest in algebraic geometry and topology (see \cite{Discriminants}, Section 7.3.D for further background).

Moreover, generic tropical oriented matroids are closely related to \emph{matching fields} from the study of maximal minors of a matrix (\cite{MaxMinors, Fields}), which in turn motivated the construction of several combinatorial objects with different axiomatics, including \emph{matching ensembles} (\cite{Ensembles}), \emph{tope arrangements} (\cite{FieldLattice}), and \emph{trianguloids} (\cite{Trianguloids}). Notably, these objects all involve some variation of the \emph{linkage axiom} from matching fields, which has similarities to the exchange axiom for ordinary matroids.

The goal of this paper is two-fold. First, we collect and synthesize the existing literature on the combinatorial properties of the aforementioned objects, and present them in a more consistent fashion for ease of understanding and comparison (\cref{sec:prelim}). Second, we extend our current understanding of the axiomatics of these objects by filling in several missing connections, most notably: \begin{itemize}
    \item Extended tope arrangements and pre-trianguloids are cryptomorphic objects (\cref{sec:operations}),
    \item A corrected proof showing that matching ensembles and generic tropical oriented matroids are cryptomorphic objects (\cref{sec:equivaxiom}), and
    \item Some new properties of fine mixed subdivisions of dilated simplices (\cref{sec:treelinkage}).
\end{itemize}
As a consequence, we obtain many more equivalent axiom systems that define a generic tropical oriented matroid, many of which can be phrased as some version of the linkage axiom (summarized in \cref{sec:bigpicture}).

\section{Preliminaries}\label{sec:prelim}


Let $n$ and $d$ be positive integers. Throughout this paper, we will use a bar for indices $\bar{j} \in [\bar{d}]$ and index sets $\bar{J}\subseteq [\bar{d}]$ that correspond to axes of $\mathbb{R}^d$, to avoid confusion with indices $i\in [n]$ and index sets $I\subseteq [n]$.\footnote{This notation is reversed from the one used in \cite{Trianguloids}.}

Let $\Delta^{d-1}$ denote the $(d-1)$-dimensional standard unit simplex, whose vertices are the standard unit vectors $e_{\bar{1}}, e_{\bar{2}}, \dots, e_{\bar{d}} \in \mathbb{R}^d$. 

\subsection{Fine mixed subdivisions of a dilated simplex}

Mixed subdivisions are a generalization of zonotope tilings that apply to any Minkowski sums of polytopes.

\begin{definition}[\cite{CayleyTri}, Section 1.2]
    Given a Minkowski sum $P = \sum_{i=1}^n P_i$, a \emph{Minkowski cell} is a full-dimensional polytope $C$ of the form $\sum_{i=1}^n C_i$, where the vertices of $C_i$ is a nonempty subset of that of $P_i$. It is \emph{fine} if it cannot be subdivided into smaller Minkowski cells.

    A polyhedral subdivision of $P$ into (fine) Minkowski cells is called a \emph{(fine) mixed subdivision} of $P$ if for any two cells $C = \sum_{i=1}^n C_i$ and $C' = \sum_{i=1}^n C'_i$, we have $C \cap C' = \sum_{i=1}^n (C_i \cap C'_i)$, and $C_i \cap C'_i$ is a common face of $C_i$ and $C'_i$. (In other words, $C$ and $C'$ \emph{intersect properly as Minkowski cells}.)
\end{definition}

In this paper we are solely interested in the case where $P_i = \Delta^{d-1}$ for $i = 1, \dots, n$, so $P$ is a dilated simplex $n\Delta^{d-1}$. In this case, every Minkowski cell $C$ can be encoded via a subgraph $G(C)$ of the complete bipartite graph $K_{n, d}$ (whose vertex set is $[n] \sqcup [\bar{d}])$, where $(i, \bar{j})$ is an edge of $G(C)$ if and only if $e_{\bar{j}}$ is a vertex of $C_i$. We will refer to $[n]$ as the \emph{left side} and $[\bar{d}]$ as the \emph{right side}. 

\begin{remark}
    It is easy to label each face $F$ of a Minkowski cell $C$ with a unique Minkowski sum decomposition $\sum_{i=1}^n F_i$, where each $F_i$ is a (nonempty) face of $C_i$. In a mixed subdivision, the condition in the definition ensures that a face shared by two or more cells receives the same label no matter which cell we choose.

    In the case where $P = n\Delta^{d-1}$, we can also represent each face $F$ with a subgraph $G(F)$. It is not difficult to see that $F$ is a face of $C$ if and only if $G(F)\subseteq G(C)$ (see also \cite{TriangulationTOM}, Remark 2.3).
\end{remark}

The definition of fine mixed subdivisions of $n\Delta^{d-1}$ can now be translated into purely combinatorial conditions on a collection of subgraphs:

\begin{definition}
    Two acyclic subgraphs $G$ and $G'$ of $K_{n, d}$ are \emph{compatible} with each other if whenever they both contain a perfect matching between two sets of vertices $I\subseteq [n]$ and $\bar{J}\subseteq [\bar{d}]$, the two matchings are the same.
\end{definition}

There is an equivalent formulation for compatibility which will be used for some proofs. 

\begin{definition}
    For two forests $G$ and $G'$ of $K_{n,d}$, let the \emph{compatibility graph} $U(G, G')$ be the directed graph on $[n]\sqcup [\bar{d}]$ consisting of all edges in $G$ directed from $[n]$ to $[\bar{d}]$ and all edges in $G'$ directed from $[\bar{d}]$ to $[n]$.\footnote{A similar graph called \emph{comparability graph} $CG_{G, G'}$ is defined in \cite{TropicalOM}, Definition 3.2, which is defined on the vertex set $[\bar{d}]$ and for not necessarily acyclic graphs. This graph give rise to the same compatibility condition when specialized to forests.} 
\end{definition}

\begin{lemma}[\cite{Permutohedra}, Lemma 12.6 and \cite{Trianguloids}, Lemma 6.1]    
    Two forests are compatible if and only if their compatibility graph does not contain a simple directed cycle of length at least 4. We call such a cycle a \emph{nontrivial cycle}.
\end{lemma}

\begin{theorem}[\cite{FlagArr}, Proposition 7.2]\label{thm:FMSaxioms}
    A collection $\mathcal{G}$ of subgraphs of $K_{n, d}$ encodes a fine mixed subdivision of $n\Delta^{d-1}$ if and only if the following holds: \begin{enumerate}
        \item Each subgraph $G\in \mathcal{G}$ is a spanning tree (this is the equivalent condition for $G$ to encode a fine Minkowski cell).
        \item \textnormal{(Tree linkage)} For each tree $G\in \mathcal{G}$ and an internal edge $e\in G$ (i.e.\ an edge that is not adjacent to a leaf vertex), there exists another tree $G' \in \mathcal{G}$ and an edge $e'\in G'$ such that $G\setminus e = G' \setminus e'$.
        \item \textnormal{(Compatibility)} Any two trees of $\mathcal{G}$ are compatible with each other.
    \end{enumerate}
\end{theorem}

We note that the $G'$ in the tree linkage axiom is necessarily unique:

\begin{lemma}\label{lem:adjtrees}
    Suppose there are two distinct spanning trees $G, G'$ of $K_{n,d}$ and two edges $e\in G$ and $e'\in G'$ such that $G\setminus e = G' \setminus e' = H$. Let the two connected components of $H$ be $H^{(1)} = I^{(1)}\sqcup \bar{J}^{(1)}$ and $H^{(2)} = I^{(2)}\sqcup \bar{J}^{(2)}$. The two trees $G$ and $G'$ are compatible if and only if either $e$ connects $I^{(1)}$ with $\bar{J}^{(2)}$ and $e'$ connects $I^{(2)}$ with $\bar{J}^{(1)}$, or vice versa. In particular, this also means that $e$ and $e'$ do not share vertices.
\end{lemma}

\begin{proof}
    Assume for the sake of contradiction and without loss of generality that both $e$ and $e'$ connects $I^{(1)}$ and $\bar{J}^{(2)}$. Suppose that $e = (i, \bar{j})$ and $e' = (i', \bar{j}')$ where $i, i'\in I^{(1)}$, $\bar{j}, \bar{j}'\in \bar{J}^{(2)}$, and at least one of $i\neq i'$ and $\bar{j}\neq \bar{j}'$ is true. The compatibility graph $U(G, G')$ contains bi-directional edges for all edges in $H^{(1)}$ and $H^{(2)}$, and hence contains a nontrivial cycle of the form $i\xrightarrow{e} \bar{j}\xrightarrow{H^{(2)}} \bar{j}' \xrightarrow{e'} i' \xrightarrow{H^{(1)}} i$, contradicting compatibility.

    On the other hand, observe that if (without loss of generality) $e$ connects $I^{(1)}$ with $\bar{J}^{(2)}$ and $e'$ connects $I^{(2)}$ with $\bar{J}^{(1)}$, then both edges between $H^{(1)}$ and $H^{(2)}$ in $U(G, G')$ are directed from $H^{(1)}$ to $H^{(2)}$, and hence it cannot contain a nontrivial cycle.
\end{proof}

As a result, there cannot be three or more distinct spanning trees in $\mathcal{G}$ that all contain $H$ as a subgraph, implying that $G'$ is unique.

\begin{remark}
    Geometrically speaking, the tree linkage axiom says that for any facet of a cell that does not lie on the boundary of $n\Delta^{d-1}$, there is (uniquely) another cell that shares the facet. The compatibility axiom says that any two cells intersect properly as Minkowski cells.
\end{remark}

\begin{figure}[h!]
    \begin{center}
    \begin{minipage}{0.45\textwidth}
    \includegraphics{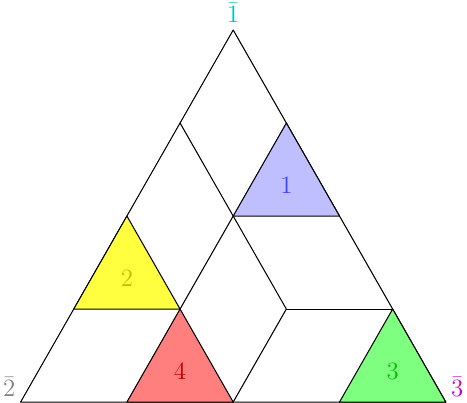}
    \end{minipage}
    \qquad
    \begin{minipage}{0.45\textwidth}
    \includegraphics{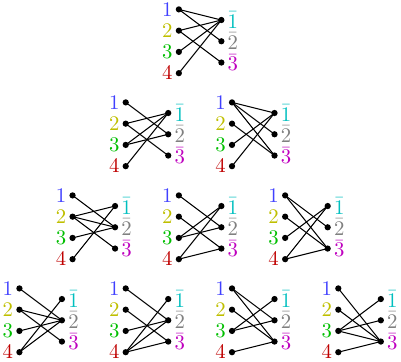}
    \end{minipage}

    \end{center}

    \caption{A fine mixed subdivision of $4\Delta^2$ and the corresponding collection of trees. The vertices of $4\Delta^2$ are labeled with the corresponding index of $[\bar{d}]$ and the unit simplices in the subdivision are labeled with the corresponding unique index of $[n]$ for which the label is $\Delta^2$.}\label{fig:FMSexample}
\end{figure}

The fine mixed subdivisions of $n\Delta^{d-1}$ are also in bijection with triangulation of products of two simplices, via what is known as \emph{the Cayley trick}:

\begin{theorem}[\cite{CayleyZono}, Theorem 3.1]
    The mixed subdivisions of $n\Delta^{d-1}$ are in bijection with polyhedral subdivisions of $\Delta^{n-1} \times \Delta^{d-1}$ where the vertices of every cell is a subset of the vertices of $\Delta^{n-1} \times \Delta^{d-1}$, and the fine mixed subdivisions biject to triangulations (i.e.\ subdivisions into simplices).
\end{theorem}

\begin{remark}
    Due to the apparent symmetry between $[n]$ and $[\bar{d}]$ in $\Delta^{n-1} \times \Delta^{d-1}$, this bijection also gives a bijection between (fine) mixed subdivisions of $n\Delta^{d-1}$ and $d\Delta^{n-1}$. While it is useful to keep this symmetry in mind, we will intentionally avoid using $\Delta^{n-1} \times \Delta^{d-1}$ or $d\Delta^{n-1}$ for the rest of this paper to prevent unnecessary confusion between the role of the two sets of indices. Note that these objects all have the same combinatorial characterization given by \cref{thm:FMSaxioms}.
\end{remark}

We conclude this sub-section with one important property of fine mixed subdivisions of $n\Delta^{d-1}$.

\begin{definition}[\cite{Trianguloids}, Definition 4.1]
    For a subgraph $G$ of $K_{n, d}$, let the \emph{left degree vector} $LD(G)$ be the $n$-tuple $(v_1, \dots, v_n)$ of nonnegative integers, where $v_i$ is the degree of vertex $i$ in $G$. Define the \emph{right degree vector} $RD(G)$ analogously (as a $d$-tuple of nonnegative integers indexed by $[\bar{d}]$).

    In addition, the graph $G$ is called a \emph{left semi-matching} if $LD(G) = \mathbf{1}_{[n]}$, a \emph{right semi-matching} if $RD(G) = \mathbf{1}_{[\bar{d}]}$, and a \emph{partial matching} if all vertices have degree $0$ or $1$. We also define \emph{partial left (resp. right) semi-matchings} as graphs whose vertices in $[n]$ (resp. $[\bar{d}]$) all have degree $0$ or $1$.
\end{definition}

\begin{lemma}\label{lem:compat}
    \begin{enumerate}
        \item[(a)] If two distinct acyclic subgraphs $G$ and $G'$ of $K_{n,d}$ have both the same left and right degree vectors, then they are not compatible with each other.
        \item[(b)] \textnormal{(\cite{Permutohedra}, Lemma 12.8)} If two distinct spanning trees $T$ and $T'$ of $K_{n,d}$ have \emph{either} the same left degree vector \emph{or} the same right degree vector, then they are not compatible with each other.
    \end{enumerate}
\end{lemma}

\begin{proof}[Proof of (a)]
    Consider the graph $U(G, G')$ and remove all bi-directional edges from it. Since $G$ and $G'$ have the same degree vectors, each vertex of the resulting graph have the same in-degree and out-degree and the graph is non-empty, and hence contains a nontrivial cycle. This means that $G$ and $G'$ are not compatible.
\end{proof}

\begin{definition}
    For a fine Minkowski cell $C$ of $n\Delta^{d-1}$, let $LD^-(C) = LD(G(C)) - \mathbf{1}_{[n]}$ and $RD^-(C) = RD(G(C)) - \mathbf{1}_{[\bar{d}]}$. These are referred to as the \emph{reduced left/right degree vectors}.
\end{definition}

\begin{proposition}[\cite{Permutohedra}, Lemma 14.9]
    Each fine Minkowski cell $C$ of $n\Delta^{d-1}$ contains exactly one lattice simplex that is a translated copy of $\Delta^{d-1}$, whose vertices are $\{RD^-(C) + e_{\bar{j}} \mid \bar{j}\in [\bar{d}]\}$. The interior of $C$ is disjoint from all other such lattice simplices.
\end{proposition}

We will refer to $RD^-(C)$ as the \emph{position} of the cell $C$, and the simplex $\text{conv}(\{RD^-(C) + e_{\bar{j}} \mid \bar{j}\in [\bar{d}]\})$ as the \emph{base simplex} of $C$. It is not difficult to see that all vertices of the base simplex are also vertices of $C$.

\begin{theorem}[\cite{Permutohedra}, Theorem 12.9]\label{thm:treedegrees}
    Over all cells $C$ of a fine mixed subdivision of $n\Delta^{d-1}$, the set of values of $LD^-(C)$ is exactly the set of lattice points in $(d-1)\Delta^{n-1}$, and the set of values of $RD^-(C)$ is exactly the set of lattice points in $(n-1)\Delta^{d-1}$. Moreover, each $LD^-(C)$ and $RD^-(C)$ is represented by exactly one cell.
\end{theorem}

\begin{remark}\label{rem:oneaxiom}
    One immediate consequence of this theorem is that any fine mixed subdivision of $n\Delta^{d-1}$ contains exactly $\binom{n+d-2}{n-1}$ cells. If we know that there are exactly this many spanning trees in a collection $\mathcal{G}$, then \emph{either one} of tree linkage axiom or compatibility axiom of \cref{thm:FMSaxioms} is sufficient in showing that $\mathcal{G}$ encodes a fine mixed subdivision. A proof of the former can be found in the proof of \cite{Trianguloids}, Theorem 3.6, and the latter is \cite{Ensembles}, Proposition 2.7.
\end{remark}

\subsection{Generic tropical oriented matroids}

The definition of a \emph{tropical oriented matroid} closely resembles the vector axioms of an oriented matroid.

\begin{definition}[\cite{TropicalOM}, Definition 3.1 and 5.2]
    A \emph{$(n,d)$-type} $T$ is an $n$-tuple $(T_1, \dots, T_n)$ of nonempty subsets of $[\bar{d}]$, and a \emph{$(n,d)$-semitype} is such a tuple of (possibly empty) subsets of $[\bar{d}]$.
\end{definition}

\begin{remark}\label{rem:trophyperplane}
    A \emph{tropical hyperplane} can be viewed as a translated copy of the normal fan of $\Delta^{d-1}$. Each tropical hyperplane hence creates $d$ closed full-dimensional ``sectors'', and divides the space into $2^d-1$ open ``faces'', which are in easy correspondence with non-empty subsets of $[\bar{d}]$: Each face corresponds to the set of sectors that it is in. For an arrangement of $n$ tropical hyperplanes, each type records one possible combination of $n$ faces (one from each tropical hyperplane) that a point in the space can simultaneously reside in.

    By drawing the Poincar\'{e} dual of each fine Minkowski cell in a fine mixed subdivision of $n\Delta^{d-1}$, one can obtain an arrangement of $n$ \emph{tropical pseudo-hyperplanes}, each of which is a tropical hyperplane under a mild piecewise-linear homeomorphism. See \cite{TropoRep}, Definition 6.4 for the precise definition of a tropical pseudo-hyperplane arrangement.
\end{remark}

\begin{figure}[h!]
    \centering
    \includegraphics{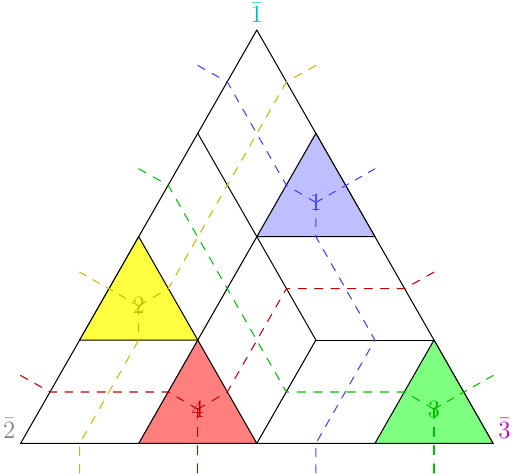}
    \caption{A tropical pseudo-hyperplane arrangement obtained from the fine mixed subdivision in \cref{fig:FMSexample}.}
\end{figure}

Similar to fine Minkowski cells, each $(n, d)$-type or semitype $T$ can be encoded via a subgraph $G(T)$ of $K_{n, d}$, where $(i, \bar{j})$ is an edge of $G(T)$ if and only if $\bar{j}\in T_i$. The graph $G$ corresponds to a type if and only if $LD(G)_i \geq 1$ for every $i\in [n]$.

Here we only provide a definition for \emph{generic} oriented matroids, which encodes generic tropical pseudo-hyperplane arrangements. All of the following definitions and results have been rephrased for this special case.

\begin{definition}[\cite{TriangulationTOM}, Definition 3.8]
    A collection $\mathcal{T}$ of $(n, d)$-types forms a \emph{generic $(n, d)$-tropical oriented matroid} if all types correspond to acyclic subgraphs of $K_{n, d}$, and \begin{enumerate}
        \item (Boundary) For every $\bar{j}\in [\bar{d}]$, the type $(\{\bar{j}\}, \dots, \{\bar{j}\})\in \mathcal{T}$.
        \item (Surrounding) If $T\in \mathcal{T}$ and $T'$ is a type such that $G(T')\subseteq G(T)$, then $T'\in \mathcal{T}$. ($T'$ is called a \emph{refinement} of $T$.)
        \item (Compatibility/Comparability) For any two types $U, V\in \mathcal{T}$, $G(U)$ and $G(V)$ are compatible with each other.
        \item (Elimination) For any two types $U, V\in \mathcal{T}$ and $i\in [n]$, there exists a type $W\in \mathcal{T}$ such that $W_i = U_i \cup V_i$ and $W_{i'} \in \{U_{i'}, V_{i'}\}$ for all $i'\in [n]\setminus \{i\}$.\footnote{In the non-generic case the condition says $W_{i'} \in \{U_{i'}, V_{i'}, U_{i'}\cup V_{i'}\}$, but from the surrounding axiom we see that the third possibility can be reduced to either of the first two in the generic case.} We say that $W$ is an \emph{elimination between $U$ and $V$ at the $i$-th coordinate}.
    \end{enumerate}
\end{definition}

\begin{definition}[\cite{Ensembles}, Section 3]
    A type $T$ of a generic tropical oriented matroid is a \emph{tree-type}\footnote{Also called a \emph{vertex} (\cite{TropicalOM}, Definition 4.1) due to it corresponding to a single point in the corresponding \emph{tropical pseudo-hyperplane arrangement}, which we will avoid using to prevent confusion with vertices of polytopes.} if $G(T)$ is a spanning tree, and is a \emph{tope} if $LD(G(T)) = \mathbf{1}_{[n]}$ (i.e.\ $G(T)$ is a \emph{left semi-matching}). 
\end{definition}

For convenience, we will sometimes shorten each coordinate of a type to a string of elements of $[\bar{d}]$ (e.g. $(\bar{2}\bar{3}, \bar{3}, \bar{1}\bar{2}, \bar{3})$). In the case of topes, we may further shorten it to a string of $n$ elements of $[\bar{d}]$ (e.g. $\bar{2}\bar{3}\bar{2}\bar{1}$).

\begin{theorem}[\cite{TropicalOM}, Theorem 4.4 and 4.6]\label{thm:TOMtreetope}
    A generic tropical oriented matroid can be completely determined by either the set of all tree-types or the set of all topes. In particular, a type $T$ is in the matroid if and only if either: \begin{enumerate}
        \item[(a)] $G(T)$ is a subgraph of $G(T')$ for a tree-type $T'$, or
        \item[(b)] $G(T)$ is compatible with $G(T')$ for all topes $T'$.
    \end{enumerate}
\end{theorem}

\begin{theorem}[\cite{TropicalOM}, Theorem 6.2 and \cite{TriangulationTOM}, Proposition 4.2]
    By identifying the tree-types with the fine Minkowski cells via the graph encoding, there is a canonical bijection between generic $(n,d)$-tropical oriented matroids and fine mixed subdivisions of $n\Delta^{d-1}$.
\end{theorem}

\begin{remark}\label{rem:topes}
    Under this bijection, the topes of a generic tropical oriented matroid correspond to lattice points in $n\Delta^{d-1}$ (i.e.\ vertices of fine Minkowski cells). Therefore, there is exactly one tope $T_v$ with $RD(G(T_v)) = v$ for every lattice point $v\in n\Delta^{d-1}$. 
\end{remark}

\begin{figure}[h!]
    \centering
    \includegraphics{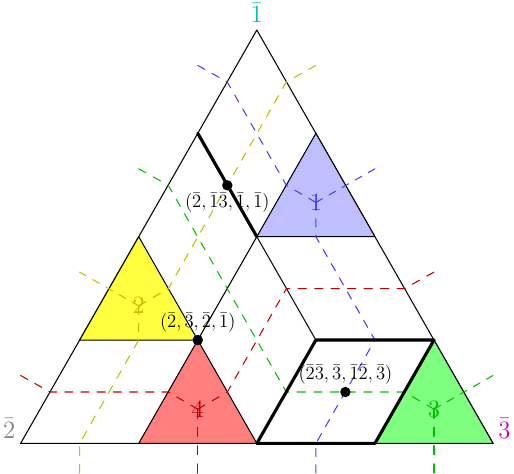}
    \caption{Three examples of types in the generic tropical oriented matroid and their corresponding cells/faces in the corresponding fine mixed subdivision in \cref{fig:FMSexample}. The one corresponding to a lattice point is a tope and the one corresponding to a rhombus is a tree-type.}
\end{figure}

The bijection also extends to \emph{all} tropical oriented matroids and \emph{all} mixed subdivisions, as well as \emph{all} tropical pseudo-hyperplane arrangements. See \cite{TropoRep} for more details.

For easier comparison with other objects defined later, it is also useful to define a slightly modified version of generic tropical oriented matroid.

\begin{definition}[\cite{TropicalOM}, Definition 5.2]
    An \emph{extended generic $(n, d)$-tropical oriented matroid} is a collection $\tilde{\mathcal{T}}$ of $(n,d)$-\emph{semi}types satisfying the same axioms as a generic tropical oriented matroid (replacing all instances of ``type'' with ``semitype''). In such a matroid, the semitypes that have no empty sets in any coordinate will be called \emph{honest types}.

    Given a generic tropical oriented matroid $\mathcal{T}$, its \emph{completion} is an extended generic tropical oriented matroid consists of all semitypes obtained by replacing some coordinates of a type with empty sets. The reverse process (of removing all non-type semitypes) is called \emph{reduction}.
\end{definition}

\begin{remark}
    It is clear that the reduction of an extended generic tropical oriented matroid is a generic tropical oriented matroid. The completion trivially preserves boundary, surrounding, and comparability axioms. To show that the completion also satisfies elimination between two semi-types $U$ and $V$, we can pass to two honest types $U^+$ and $V^+$ in the original non-extended matroid, eliminate between them to get a type $W^+$, and then refining it to $W$.
\end{remark}

The fact that completion and reduction give a full bijection between extended and non-extended generic tropical oriented matroids follows from the following lemma:

\begin{lemma}\label{lem:refinehonest}
    In an extended generic tropical oriented matroid $\tilde{\mathcal{T}}$, any semitype $T\in \tilde{\mathcal{T}}$ is a refinement of an honest type $T^+\in \tilde{\mathcal{T}}$. 
\end{lemma}

\begin{proof}
    Assume without loss of generality that the first $m$ coordinates of $T$ are non-empty and the last $n-m$ coordinares are empty (in other words, we assume that the \emph{support} of $T$ is $[m]$). If $m = 0$ or $m = n$ there is nothing to prove, so we assume $1\leq m \leq n-1$. Using the same proof as \cite{TropicalOM}, Theorem 4.6, we can use a series of elimination between semitypes supported on $[m]$ to obtain a semitype $T'$ for which $T$ is a refinement, such that $G(T')$ is a spanning tree on $[m]\sqcup [\bar{d}]$. 

    Since $\tilde{\mathcal{T}}$ contains a (non-extended) generic tropical oriented matroid and hence encodes a fine mixed subdivision, by \cref{thm:treedegrees} there is a unique tree-type $T_0$ such that $LD(G(T_0)) = LD(G(T')) + \mathbf{1}_{[n]\setminus[m]}$. Since $T'$ is compatible with $T_0$, by \cref{lem:compat}(b) we must have $G(T')\subseteq G(T_0)$, and hence $T^+ = T_0$ is a desired honest type.
\end{proof}

\subsection{Linkage matching fields and tope arrangements}

\emph{Matching fields} are first studied by Strumfels and Zelevinsky in \cite{MaxMinors} as a way to understand the Newton polytope of the product of all maximal minors of a rectangular matrix.

\begin{definition}[\cite{Fields}]
    For positive integers $n' \geq d$, a \emph{$(n', d)$-matching field} $\mathcal{M}$ is a collection of bijections (or \emph{right semi-matchings}) $M_\sigma$ between $\sigma$ and $[\bar{d}]$, one for each $d$-element subset $\sigma\subseteq [n']$. 

    For convenience, we will also treat $M_\sigma$ as a perfect matching $G(M_\sigma)$ between $\sigma$ and $[\bar{d}]$ in $K_{n', d}$.
\end{definition}

If one write each $M_\sigma$ as an $n'\times d$ indicator matrix, it is not difficult to see that the sum of all matrices of a particular matching field corresponds to a lattice point in the Newton polytope. It was observed in \cite{MaxMinors} that the vertices of this polytope (corresponding to \emph{coherent} matching fields) must satisfy the \emph{linkage axiom}. There are many equivalent ways to state the linkage axiom:

\begin{theorem}[\cite{Fields}, Theorem 2 and Lemma 10]\label{thm:linkageequiv}
    For a matching field $\mathcal{M}$, the following conditions are equivalent: \begin{enumerate}
        \item \textnormal{(Strong linkage)} For any $(d+1)$-element subset $\tau\subseteq [n']$, the union of matchings $G(M_\sigma)$ over all $\sigma\subseteq \tau$ (treated as subgraphs of $K_{n', d}$) is a spanning tree $\mathbb{T}_{\tau}$ of $\tau \sqcup [\bar{d}]$ whose vertices in $[\bar{d}]$ all have degree 2.
        \item \textnormal{(Weak linkage)} For any $(d+1)$-element subset $\tau\subseteq [n']$ and $\bar{j}\in [\bar{d}]$, there exists two elements $i, i'\in \tau$ such that $G(M_{\tau\setminus \{i\}})$ and $G(M_{\tau\setminus \{i'\}})$ differ only in the neighbor of $\bar{j}$.
        \item \textnormal{(Three-element linkage)} For any $(d+1)$-element subset $\tau\subseteq [n']$ and $i_1, i_2, i_3\in \tau$, if $M_{\tau\setminus\{i_1\}}(i_2) \neq M_{\tau\setminus\{i_3\}}(i_2)$ then $M_{\tau\setminus\{i_1\}}(i_3) = M_{\tau\setminus\{i_2\}}(i_3)$.
        \item \textnormal{(Exchange)} For any two distinct $d$-element subsets $\sigma, \sigma'\subseteq [n']$, there exist $i\in \sigma\setminus\sigma'$ and $i'\in \sigma'\setminus \sigma$ such that $M_{\sigma}(i) = M_{\sigma'}(i') = \bar{j}$ for some $\bar{j}\in [\bar{d}]$, and replacing the edge $(i, j)$ in $G(M_\sigma)$ with the edge $(i', j)$ gives the matching $G(M_{\sigma\setminus\{i\}\cup \{i'\}})$.
        \item \textnormal{(Elimination)} For any two distinct $d$-element subsets $\sigma, \sigma'\subseteq [n']$ and $i\in \sigma\cap \sigma'$ such that $M_{\sigma}(i)\neq M_{\sigma'}(i)$, there exists a $d$-element subset $\sigma'' \subseteq \sigma \cup\sigma' \setminus \{i\}$ such that $G(M_{\sigma''}) \subseteq G(M_{\sigma}) \cup G(M_{\sigma'})$.
    \end{enumerate}
\end{theorem}

A matching field that satisfy any of the conditions above is called \emph{linkage}. Unless otherwise specified, we will use condition (1) as the default linkage axiom from now on. The tree $\mathbb{T}_{\tau}$ will be called a \emph{linkage covector}.\footnote{Also called a \emph{linkage tree} in \cite{MaxMinors}, or a \emph{linkage pd-graph} in \cite{TropicalToOM}.}

Loho and Smith (\cite{FieldLattice}) generalized matching fields to allow for elements of $[\bar{d}]$ to appear multiple times:

\begin{definition}[\cite{FieldLattice}, Definition 2.1 and 3.6]
    Given a $d$-tuple $v = (v_{\bar{1}}, \dots, v_{\bar{d}})$ of positive integers, let $k = \sum_{\bar{j}\in[\bar{d}]} v_{\bar{j}}$. An \emph{$(n', d)$-tope field $\mathcal{M}$ of type $v$} is a collection of maps (or \emph{partial topes}) $M_{\sigma}$ from $\sigma$ to $[\bar{d}]$, one for each $k$-element subset of $\sigma\subseteq[n']$, such that $RD(G(M_{\sigma})) = v$. We say that $k$ is the \emph{thickness} of $\mathcal{M}$; if $k = n'$, we say that the tope field is \emph{maximal}, which consists of a single (full) tope.

    A tope field of type $v$ is \emph{linkage} if for any $(k+1)$-element subset $\tau\subseteq [n']$, the union of $G(M_\sigma)$ over all $\sigma\subseteq \tau$ is a tree $\mathbb{T}_{\tau}(v)$ whose right degree vector is equal to $v + \mathbf{1}_{\bar{d}}$.
\end{definition}

Note that a (linkage) matching field is the same as a (linkage) tope field of type $\mathbf{1}_{\bar{d}}$. Through this generalization, one can construct a collection of compatible (maximal) topes from a linkage matching field.

\begin{definition}[\cite{FieldLattice}, Definition 3.30]
    An \emph{$(n', d)$-tope arrangement} $\mathcal{T}$ is a collection of pairwise compatible topes $T_v$, one for every lattice point $v \in n'\Delta^{d-1}\cap (\mathbb{Z}^+)^{d}$, such that $RD(G(T_v)) = v$ for each $v$. We call $v$ the \emph{position} of $T_v$.
\end{definition}

\begin{theorem}[\cite{FieldLattice}, Theorem 3.32]\label{thm:topearr}
    Linkage $(n',d)$-matching fields and $(n',d)$-tope arrangements are cryptomorphic objects.
\end{theorem}

The process of obtaining a tope arrangement from a linkage matching field is called \emph{iterated amalgamation}, which we will describe in \cref{sec:operations}.

\subsection{Pointed matching fields and matching ensembles}

While there are many similarities between topes from a generic tropical oriented matroid (c.f. \cref{rem:topes}) and tope arrangements from linkage matching fields, we must note that the topes in a tope arrangement are only defined for ``interior'' lattice points of a dilated simplex. We can resolve this discrepancy by working with \emph{pointed} matching fields.

\begin{definition}[\cite{MaxMinors}, Example 1.4]
    For positive integers $n$ and $d$, a \emph{pointed $(n+d, d)$-matching field} is a matching field whose left indices are $[n]\sqcup [\underline{d}]$, such that whenever some $\underline{j} \in [\underline{d}]$ is in a $d$-element subset $\sigma$, we have $M_\sigma(\underline{j}) = \bar{j}$.
\end{definition}

Given a matching $M_{\sigma}$ between $\sigma \subseteq [n]\sqcup [\underline{d}]$ and $[\bar{d}]$, let $\sigma' = \sigma\cap [n]$, $\underline{\sigma} = \sigma\cap[\underline{d}]$, and let $\bar{\sigma}$ be the subset of $[\bar{d}]$ corresponding to $\underline{\sigma}$. If we remove all elements of $[\bar{d}]$ from $M_{\sigma}$, we obtain a matching between $I = \sigma'$ and $\bar{J} = [\bar{d}]\setminus \bar{\sigma}$. As $\sigma$ ranges over all $d$-element subsets, $(I, \bar{J})$ ranges over all pairs of subsets of $[n]$ and $[\bar{d}]$ of equal size. Therefore, pointed matching fields are cryptomorphic to \emph{matching stacks}, where we record a matching for every pair of equal-size subsets. 

\begin{definition}[\cite{FieldLattice}, Definition 5.1]
    For positive integers $n$ and $d$, a \emph{$(n,d)$-matching stack}\footnote{Also called a \emph{matching field} (\cite{Ensembles}, Definition 4.1), which we will avoid using to prevent obvious collision in names.} $\mathcal{M}$ is a collection of bijections (or \emph{partial matchings}) $M_{I, \bar{J}}$ between $I$ and $\bar{J}$ for every pair of subsets $I\subseteq [n]$ and $\bar{J}\subseteq [\bar{d}]$ such that $|I| = |\bar{J}|$. 

    A matching stack is a \emph{matching ensemble} if the following conditions hold: \begin{enumerate}
        \item (Closure) Given $I'\subseteq I$ and $\bar{J}'\subseteq \bar{J}$, if $G(M_{I, \bar{J}})$ contains a matching between $I'$ and $\bar{J}'$, then $G(M_{I', \bar{J}'}) \subseteq G(M_{I, \bar{J}})$. (Equivalently, all matchings are pairwise compatible.)
        \item (Left linkage) Given subsets $I\subseteq [n]$ and $\bar{J}\subseteq [\bar{d}]$ such that $|I| = |\bar{J}| + 1$, the union of $G(M_{I', \bar{J}})$ over all $I'\subset I$ is a tree $\mathbb{T}_{I, \bar{J}}$ whose right degree vector is equal to $\mathbf{2}_{\bar{J}}$.
        \item (Right linkage) Given subsets $I\subseteq [n]$ and $\bar{J}\subseteq [\bar{d}]$ such that $|I| + 1 = |\bar{J}|$, the union of $G(M_{I, \bar{J}'})$ over all $\bar{J}'\subset \bar{J}$ is a tree $\bar{\mathbb{T}}_{I, \bar{J}}$ whose left degree vector is equal to $\mathbf{2}_{I}$.
    \end{enumerate}

    The trees $\mathbb{T}_{I, \bar{J}}$ and $\bar{\mathbb{T}}_{I, \bar{J}}$ are called \emph{(left/right) linkage covectors}. For convenience, we also define a \emph{matching left-semi-ensemble} to be a matching stack that only satisfies the first two conditions.
\end{definition}

The process of obtaining a matching stack from a pointed matching field will be referred to as \emph{reduction} process, and the inverse as the \emph{completion} (also see \cite{FieldLattice}, Section 5).

In the language of Newton polytopes, matching stacks correspond to taking a product of \emph{all} minors of a rectangular matrix, not just the maximal ones. We also note that the aforementioned cryptomorphism is essentially the same as the one between matroids and \emph{bimatroids} (also known as \emph{linking systems}, see \cite{Bimatroids}, Theorem 1 and \cite{LinkingSys}, Theorem 3.2).

This cryptomorphism also allow us to transfer the linkage axiom from pointed matching fields to matching stacks.

\begin{theorem}[\cite{FieldLattice}, Theorem 5.5, corrected\footnote{The original theorem statement mistakenly omitted the closure axiom.}]\label{thm:fieldstack}
    Linkage $(n+d, d)$-pointed matching fields are cryptomorphic to $(n,d)$-matching left-semi-ensembles.
\end{theorem}

\begin{proof}
    Given a linkage pointed matching field, consider a linkage covector $\mathbb{T}_{I\sqcup \underline{J}}$ where $I\subseteq [n], \underline{J}\subseteq [\underline{d}]$, and $|I \sqcup \underline{J}| = d + 1$. Due to the matching field being pointed, all vertices in $\underline{J}$ are leaves of the covector, and are only connected to their counterpart in $\bar{J}\subseteq [\bar{d}]$. This tree is a union of $d+1$ right semi-matchings, and after reduction the matchings fall into one of the following two types:\begin{enumerate}
        \item A partial matching between $I\setminus \{i\}$ and $[\bar{d}]\setminus \bar{J}$ for each $i\in I$; and
        \item A partial matching between $I$ and $[\bar{d}]\setminus \bar{J} \cup \{\bar{j}\}$ for each $\underline{j}\in \underline{J}$.
    \end{enumerate}
    If we remove the edge connected to $\bar{j}$ (say $(i, \bar{j})$) from a type-(2) partial matching $M_2 = M_{I, [\bar{d}]\setminus \bar{J} \cup \{\bar{j}\}}$, then we obtain a partial matching $M_2'$ on the same vertex set as a type-(1) partial matching $M_1 = M_{I\setminus\{i\}, [\bar{d}]\setminus\bar{J}}$. Since the union of all $d+1$ reduced matchings is the tree $T_{I\sqcup \underline{J}}$ with all vertices in $\underline{J}$ removed, we must have $M_2' = M_1$ or else the union of $M_1$ and $M_2$ is not acyclic. By repeating this argument for every $I, \underline{J}$, and $\underline{j}\in \underline{J}$, we see that the closure axiom holds for every single-edge removal, and hence in general by transitivity.

    Also, note that after reduction, all vertices of $\mathbb{T}_{I\sqcup \underline{J}}$ in $\bar{J}$ are now leaves (with the single edge coming only from the corresponding type-(2) matching). Hence, by removing all edges connected to $\bar{J}$, we obtain a new tree $\mathbb{T}_{I, [\bar{d}]\setminus\bar{J}}$. It is easy to see that this tree is the union of all type-(1) matchings with the correct right degree vector, and hence the left linkage axiom is also satisfied.

    The reverse implication (completion of a matching left-semi-ensemble is a linkage pointed matching field) can be similarly shown by running the argument above in reverse. 
\end{proof}

\begin{remark}
    From this proof, we can see that the closure and left linkage axioms can be combined into one axiom, which more directly corresponds to the linkage axiom in a linkage pointed matching field: \begin{itemize}
        \item (Extended left linkage) Given subsets $I\subseteq [n]$ and $\bar{J}\subseteq [\bar{d}]$ such that $|I| = |\bar{J}| + 1$, the union of $G(M_{I', \bar{J}})$ over all $I'\subset I$ \emph{as well as} $G(M_{I, \bar{J}'})$ over all $\bar{J}'\supset \bar{J}$ is a tree $\tilde{\mathbb{T}}_{I, \bar{J}}$ whose right degree vector is equal to $\mathbf{2}_{\bar{J}} + \mathbf{1}_{[\bar{d}]\setminus\bar{J}}$. The tree $\tilde{\mathbb{T}}_{I, \bar{J}}$ will be called an \emph{extended left linkage covector}.
    \end{itemize}
    Note that assuming closure holds, each of the additional matchings $M_{I, \bar{J}'}$ only contributes one leaf edge to the extended left linkage covector.
    
    We can also define an extended right linkage axiom for matching ensembles symmetrically. When translated back to pointed matching fields, this axiom is equivalent to following condition: \begin{itemize}
        \item (Strong inverse linkage) For any $(d-1)$-element subset $\rho\subseteq [n']$, the union of matchings $G(M_{\sigma})$ over all $\sigma \supseteq \rho$ is a forest $\bar{\mathbb{T}}_{\rho}$ whose vertices in $\rho\cap [n]$ all have degree 2 and the other vertices in $[n']$ all have degree 1.
    \end{itemize}
    It is unclear whether this property has additional implications for general matching fields.
\end{remark}

\begin{figure}[h!]
    \includegraphics{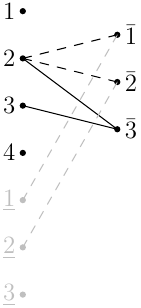}
    \caption{A linkage covector $\mathbb{T}_{\{2, 3, \underline{1}, \underline{2}\}}$ of a pointed matching field (all edges/vertices), an extended left linkage covector $\tilde{\mathbb{T}}_{\{2, 3\}, \{\bar{3}\}}$ (black edges/vertices only), and a left linkage covector $\mathbb{T}_{\{2, 3\}, \{\bar{3}\}}$ (solid edges only) of the corresponding matching left-semi-ensemble.}
\end{figure}

Under this cryptomorphism, we can obtain \emph{extended $(n, d)$-tope arrangements} from matching left-semi-ensembles in the same way as \cref{thm:topearr}, where there is a tope $T_v$ for \emph{every} lattice point $v\in n\Delta^{d-1}$, not just the ones with positive coordinates. (Note that when $n' = n+d$, the positive lattice points of $n'\Delta^{d-1}$ are just the lattice points in $n\Delta^{d-1}$ translated by $\mathbf{1}_{\bar{d}}$.) 

\begin{figure}[h!]
    \centering
    \includegraphics{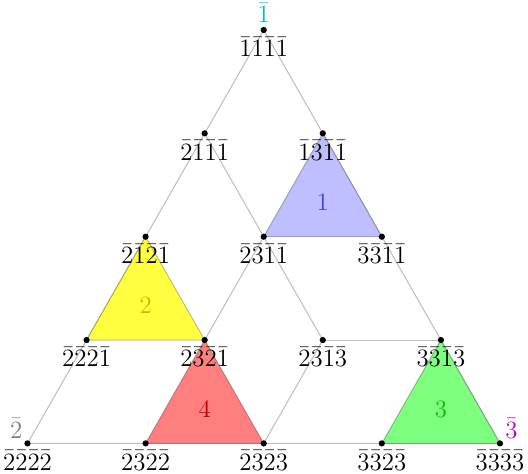}

    \caption{An extended tope arrangement obtained from the fine mixed subdivision in \cref{fig:FMSexample}.}
\end{figure}

All properties of tope arrangements can be translated into similar properties for extended tope arrangements, so whenever we cite a result about tope arrangements in later sections, it will be automatically translated into one about extended tope arrangements without proof.


Matching ensembles are first studied by Oh and Yoo (\cite{Ensembles}) as another way to encode generic tropical oriented matroids.

\begin{conjecture}[\cite{Ensembles}, Theorem 5.4]\label{thm:TOMensembles}
    Generic $(n,d)$-tropical oriented matroids are in bijection with $(n,d)$-matching ensembles.
\end{conjecture}

Given a fine mixed subdivision as encoded by a collection of trees, one can use the \emph{extraction method} (\cite{Ensembles}, Proposition 4.3) to obtain a unique matching $M_{I, \bar{J}}$ for every pair of subsets $I$ and $\bar{J}$ that is contained in some tree. (Note that the uniqueness is guaranteed by pairwise compatibility.) However, the proof that every matching ensemble comes from a fine mixed subdivision (\cite{Ensembles}, Proposition 5.3) contains a gap; see the next example. We will provide a corrected proof of this statement in \cref{sec:equivaxiom}.

\begin{example}
    Lemma 3.16 of \cite{Ensembles} claims that if there are two types $T$ and $T'$ in a (generic) tropical oriented matroid $\mathcal{T}$ and $\bar{j}, \bar{j}'\in [\bar{d}]$ such that removing all instances of $\bar{j}$ from $T$ gives the same semitype as the one from removing all instances of $\bar{j}'$ from $T'$, then $T\cup T'$ (where the union is taken coordinate-wise) is also a type of $\mathcal{T}$. We show that this is false even when there is only one instance of $\bar{j}$ and $\bar{j}'$ in $T$ and $T'$ respectively (which appears to be the only case needed in the proof of Proposition 5.3 in the same paper).

    Consider a generic $(2, 4)$-tropical oriented matroid whose tree-types are \[(\bar{1}\bar{2}\bar{3}\bar{4},\bar{1}), (\bar{2}\bar{3}\bar{4},\bar{1}\bar{2}), (\bar{3}\bar{4},\bar{1}\bar{2}\bar{3}), (\bar{4},\bar{1}\bar{2}\bar{3}\bar{4}).\] 
    By \cref{thm:TOMtreetope}(b), all other types are refinements of (at least) one of these tree-types. This matroid corresponds to the fine mixed subdivision of $2\Delta^3$ shown in \cref{fig:2d3}.

    \begin{figure}[h!]
        \centering
        \includegraphics{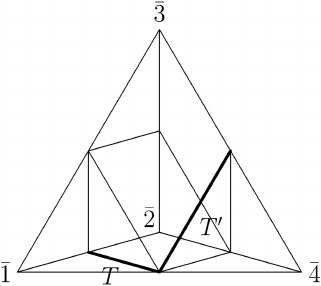}
        \caption{A fine mixed subdivision of $2\Delta^3$ with two ``adjacent'' edges but no face containing both of them.}\label{fig:2d3}
    \end{figure}

    Let $T = (\bar{2}\bar{4},\bar{1})$ and $T' = (\bar{4}, \bar{1}\bar{3})$, obtained by refining the first (or second) and the fourth (or third) tree-type, respectively. (These two types correspond to the bolded edges in the figure.) These two types are the same if one removes $\bar{2}$ from $T$ and $\bar{3}$ from $T'$, but their union $T\cup T' = (\bar{2}\bar{4}, \bar{1}\bar{3})$ is not a type as it is not compatible with the second or third tree-type (which contains a different matching between $\{1, 2\}$ and $\{\bar{2}, \bar{3}\}$).
\end{example}

\subsection{Trianguloids} Galashin et al.\ (\cite{Trianguloids}) introduced another matroid-like object called \emph{trianguloids} that are also in bijection with fine mixed subdivisions of $n\Delta^{d-1}$.\footnote{In fact, trianguloids are defined for all \emph{generalized permutohedra} that are Minkowski sums of faces of $\Delta^{d-1}$. In this paper we focus on the $n\Delta^{d-1}$ case.} While the original object is defined as a colored multi-graph, here we rephrase the definition entirely in terms of topes.

\begin{definition}[\cite{FieldLattice}, Definition 3.33]\label{def:trianguloid}
    An \emph{$(n,d)$-pre-trianguloid} $\mathcal{T}$ is a collection of \emph{topes} $T_v: [n] \to [\bar{d}]$, one for each lattice point $v\in n\Delta^{d-1}$, such that the following holds: \begin{enumerate}
        \item For each $v\in n\Delta^{d-1}$, $RD(G(T_v)) = v$ (i.e.\ $|T_v^{-1}(\bar{j})| = v_{\bar{j}}$ for each $\bar{j}\in [\bar{d}]$).
        \item (Combinatorial sector) If $v - e_{\bar{j}} = v' - e_{\bar{j}'}$ for some $v, v'\in n\Delta^{d-1}$ and $\bar{j}, \bar{j}'\in [\bar{d}]$, then $T_{v'}^{-1}(\bar{j})\subset T_{v}^{-1}(\bar{j})$. (Note that the second set is one element larger than the first.) 
        
        In other words, for each $i\in [n]$ and $\bar{j}\in [\bar{d}]$, the set $S_{i}^{(\bar{j})}: = \{v\in n\Delta^{d-1}\mid T_v(i) = \bar{j}\}$ is closed under ``moving towards $ne_{\bar{j}}$''. This set is also called a \emph{combinatorial sector}, in analogy to a sector of a tropical (pseudo-)hyperplane (\cite{FieldLattice}, Definition 3.28)
    \end{enumerate}

    The pre-trianguloid is a \emph{trianguloid} if the following also holds: \begin{enumerate}
        \item[(3)] (Hexagon axiom) For any lattice point $w\in (n+1)\Delta^{d-1}$ and $\bar{j}_1, \bar{j}_2, \bar{j}_3 \in [\bar{d}]$ such that $w_{\bar{j}_k} \geq 1$ for $k = 1,2,3$, if $T^{-1}_{w - e_{\bar{j}_1}}(\bar{j}_2) \neq T^{-1}_{w - e_{\bar{j}_3}}(\bar{j}_2)$ then $T^{-1}_{w - e_{\bar{j}_1}}(\bar{j}_3) = T^{-1}_{w - e_{\bar{j}_2}}(\bar{j}_3)$.
    \end{enumerate}
\end{definition}

Note that both extended tope arrangements and (pre-)trianguloids are both defined as a collection of topes.

\begin{theorem}[\cite{FieldLattice}, Proposition 3.34]\label{thm:extpretri}
    Extended $(n, d)$-tope arrangements are pre-trianguloids but not always trianguloids.
\end{theorem}

\begin{theorem}[\cite{Trianguloids}, Theorem 3.6]\label{thm:trianguloids}
    The $(n, d)$-trianguloids are in bijection with fine mixed subdivisions of $n\Delta^{d-1}$.
\end{theorem}

\begin{figure}[h!]
    \includegraphics{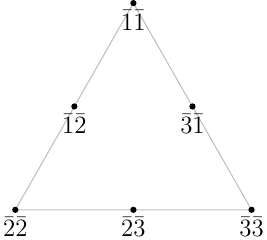}
    \caption{(Adapted from \cite{Trianguloids}, Figure 3) An extended $(2,3)$-tope arrangement (also a pre-trianguloid) that does not correspond to a fine mixed subdivision.}
\end{figure}

\section{Matroid-like Operations}\label{sec:operations}
In this section we discuss some basic operations that apply to many of the objects defined in the previous section, and use them to derive some useful properties. 

Since fine Minkowski cells, (semi-)types, left/right semi-matchings, linkage covectors, partial matchings, and (partial) topes can all be expressed as acyclic subgraphs of $K_{n, d}$, this and all following sections will be worded entirely in terms of bipartite graphs, and all objects will be treated as such. (In other words, we will drop the implicitly defined map $G$ that converts cells, types, etc.\ to graphs from the previous section.)

\subsection{Pushing and pulling} 
There is a simple combinatorial operation that connects left/right linkage covectors and (partial) left/right semi-matchings.

\begin{definition}
    Given a spanning tree $\mathbb{T}$ that is a bipartite graph between two vertex sets $I$ and $\bar{J}$, for each $i\in I$, define the \emph{$i$-push} (resp. \emph{$i$-pull}) of the tree $\mathbb{T}$ as the subgraph $\mathbb{T}^{i\rightarrow}$ (resp. $\mathbb{T}^{i\leftarrow}$ obtained by orienting every edge away (resp. towards) the vertex $i$ and taking all edges that are oriented from $\bar{J}$ to $I$. (The orientations are discarded after this operation.) Define the \emph{$\bar{j}$-push} $\mathbb{T}^{\leftarrow\bar{j}}$ and \emph{$\bar{j}$-pull} $\mathbb{T}^{\rightarrow\bar{j}}$ for each $\bar{j}\in \bar{J}$ symmetrically.
\end{definition}

\begin{figure}[h!]
    \centering
    \includegraphics{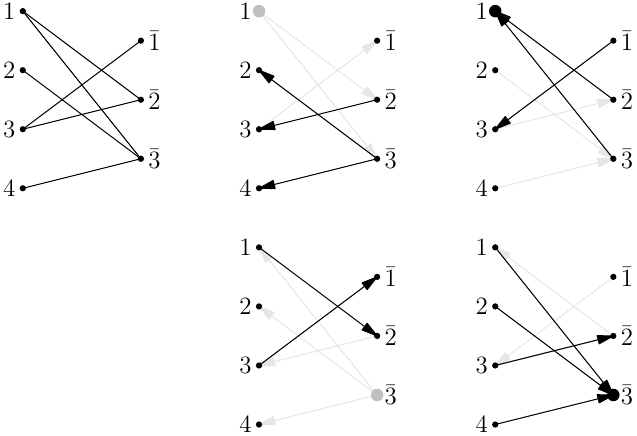}
    \caption{A spanning tree on $[4]\sqcup[\bar{3}]$ and the $1$-push, $1$-pull, $\bar{3}$-push, and $\bar{3}$-pull of the tree. The ``pushing''/``pulling'' vertex is indicated by a large gray/black dot. The arrows are only for illustrating the construction; the resulting graphs are not oriented.}
\end{figure}

\begin{remark}\label{rem:hypertree}
    The name \emph{push} and \emph{pull} might make more intuitive sense if we treat each $\mathbb{T}$ as a \emph{hypertree} with vertex set on $[\bar{d}]$ and $n$ (hyper-)edges labeled with $[n]$ (each hyper-edge connects a nonempty subset of the vertices). The $\bar{j}$-pull replaces the vertex set of each hyper-edge with the unique vertex that is closest to $\bar{j}$, and the $\bar{j}$-push replaces each set with the set of all other vertices (that are farther away from $\bar{j}$, possibly none).
\end{remark}

It is easy to verify the following properties:

\begin{lemma}\label{lem:pushpull}
    For a spanning tree $\mathbb{T}$ on $I\sqcup \bar{J}$ and for any $i\in I$ and $\bar{j}\in \bar{J}$, we have:
    \begin{enumerate}
        \item[(a)] $\mathbb{T} = \mathbb{T}^{i\rightarrow} \sqcup \mathbb{T}^{i\leftarrow} = \mathbb{T}^{\rightarrow\bar{j}} \sqcup \mathbb{T}^{\leftarrow\bar{j}}$ (in terms of the set of edges).
        \item[(b)] $LD(\mathbb{T}^{i\rightarrow}) = \mathbf{1}_{I\setminus \{i\}}, LD(\mathbb{T}^{\rightarrow\bar{j}}) = \mathbf{1}_{I}, RD(\mathbb{T}^{i\leftarrow}) = \mathbf{1}_{\bar{J}}, RD(\mathbb{T}^{\leftarrow\bar{j}}) = \mathbf{1}_{\bar{J}\setminus \{\bar{j}\}}$. In particular, the first two operations give (partial) left semi-matchings and the latter two give (partial) right semi-matchings.
        \item[(c)] If $i$ and $\bar{j}$ are connected by edge $e$, then $\mathbb{T}^{i\rightarrow} \sqcup e = \mathbb{T}^{\rightarrow\bar{j}}$ and $\mathbb{T}^{\leftarrow\bar{j}} \sqcup e = \mathbb{T}^{i\leftarrow}$.
        \item[(d)] For $i, i'\in I$, the two graphs $\mathbb{T}^{i\rightarrow}$ and $\mathbb{T}^{i'\rightarrow}$ differ by exactly one edge if and only if there is some $\bar{j}\in \bar{J}$ that is adjacent to both $i$ and $i'$. The same applies to $\mathbb{T}^{i\leftarrow}$ and $\mathbb{T}^{i'\leftarrow}$, and symmetrically for $\bar{j}, \bar{j}'\in \bar{J}$.
    \end{enumerate}
\end{lemma}

This operation has appeared in many of the previous works, though without an explicit name. We collect some of the results here, rephrased using this definition:

\begin{proposition}[\cite{Fields}, Lemma 10 and \cite{FieldLattice}, Lemma 3.10]\label{prop:ipush}
    Given a linkage covector $\mathbb{T}_{\tau}(v)$ of a linkage tope field $\mathcal{M}$ and $i\in \tau$, we can recover the partial tope for $\tau\setminus\{i\}$ via $M_{\tau\setminus\{i\}} = (\mathbb{T}_{\tau}(v))^{i\rightarrow}$.
\end{proposition}

\begin{proposition}[\cite{FieldLattice}, Definition 3.12 and Proposition 3.15]
    Given a linkage tope field of $\mathcal{M}$ of type $v$ and thickness $k < n'$ and some $\bar{j}\in [\bar{d}]$, if we take the $\bar{j}$-pull of every linkage covector $\mathbb{T}_{\tau}(v)$ over all $\tau\in \binom{[n']}{k+1}$, the collection of graphs $M_\tau = (\mathbb{T}_{\tau}(v))^{\rightarrow \bar{j}}$ defines a linkage tope field $\mathcal{M}^{+\bar{j}}$ of type $v + e_{\bar{j}}$. This is called the \emph{$\bar{j}$-amalgamation} of $\mathcal{M}$.
\end{proposition}

\begin{corollary}[Iterated amalgamation, \cite{FieldLattice}, Theorem 3.16]
    For each lattice point $v\in n'\Delta^{d-1}\cap (\mathbb{Z}^+)^d$, by performing $\bar{j}$-amalgamation $v_{\bar{j}}-1$ times for each $\bar{j}\in [\bar{d}]$ (in any order) starting from a linkage matching field $\mathcal{M}$, one obtains the unique tope with right degree $v$ that is compatible with all matchings in $\mathcal{M}$.
\end{corollary}

\begin{remark}
    Note that from \cref{lem:pushpull}(c), it is easy to see that for a partial tope $M_\tau$ in $\mathcal{M}^{+\bar{j}}$, removing any edge of the form $(i, \bar{j})$ from it gives a partial tope $M_{\tau\setminus\{i\}}$ that is in $\mathcal{M}$. In fact, the relationship between these two partial topes are essentially the same as the pair of type-(2) and type-(1) partial matchings $M_2$ and $M_1$ that appear in the proof of \cref{thm:fieldstack}. 
    
    Therefore, the set of all partial topes that arise from iteratively amalgamating a linkage matching field are closed under taking subgraphs, similar to matching ensembles and generic tropical oriented matroids. The compatibility between these topes follows from the fact that the right semi-matchings of a linkage matching field are pairwise compatible (see \cite{FieldLattice}, Proposition 3.5). The same applies to matching left-semi-ensembles.

    We also note that a similar (and ultimately equivalent) construction of an extended tope arrangement from a matching ensemble when $n\leq d$ is presented in the proof of Proposition 5.3 of \cite{Ensembles}.
\end{remark}

\begin{proposition}[\cite{Trianguloids}, Proposition 3.4]\label{prop:treeunion}
    In a pre-trianguloid, the union of all topes $T_{u + e_{\bar{j}}}$ whose positions are the vertices of a unit simplex $u + \Delta^{d-1}$ for some $u\in (n-1)\Delta^{d-1}$ is a tree $\mathbb{T}(u)$ whose right degree vector is $u + \mathbf{1}_{[\bar{d}]}$. In particular, we have $T_{u + e_{\bar{j}}} = (\mathbb{T}(u))^{\rightarrow \bar{j}}$.
\end{proposition}

\begin{remark}
    If we replace the pre-trianguloid in this proposition with an extended tope arrangement, there is a simple reason for why this proposition is true: $\mathbb{T}(u)$ is simply the linkage covector $\mathbb{T}_{[n]}(u)$ of the linkage tope field $\mathcal{M}$ of type $u$. We will show in the next sub-section that pre-trianguloids and extended tope arrangements are in fact cryptomorphic.
\end{remark}

\begin{remark}
    Note that this ``$\Delta$-linkage'' property also implies the combinatorial sector axiom: The set of neighbors of $\bar{j}$ in $T_{u + e_{\bar{j}}} = (\mathbb{T}(u))^{\rightarrow \bar{j}}$ is clearly a superset of that in $T_{u + e_{\bar{j}'}} = (\mathbb{T}(u))^{\rightarrow \bar{j}'}$ for any $\bar{j}'\neq \bar{j}$. Therefore, these two properties are equivalent.
\end{remark}

When the pre-trianguloid is a trianguloid (i.e. corresponding to a fine mixed subdivision), it is not difficult to see that these trees $\mathbb{T}(u)$ are exactly the trees encoding the fine Minkowski cells.

We mention one more consequence of the combinatorial sector axiom:

\begin{proposition}\label{prop:largetreeunion}
    Let $m$ be a positive integer, and $P = s + (m-1)\Delta^{d-1}$ where $s$ is a lattice point in $(n-m)\Delta^{d-1}$. In a pre-trianguloid, $\bigcup_{u\in P\cap \mathbb{Z}^{d}} \mathbb{T}(u)$ is a (connected) graph whose right degree vector is $s + m\cdot \mathbf{1}_{[\bar{d}]}$.
\end{proposition}

\begin{proof}
    We first show that $\bigcup_{u\in P\cap \mathbb{Z}^{d}} \mathbb{T}(u) = \bigcup_{\bar{j}\in[\bar{d}]} T_{s + me_{\bar{j}}}$. From \cref{prop:treeunion}, each tree is a union of topes of the form $T_{u + e_{\bar{j}}}$ where $u+e_{\bar{j}}\in s + m\Delta^{d-1}$. For each $\bar{j}\in[\bar{d}]$, the set of neighbors of $\bar{j}$ in each of these topes is a subset of that in $T_{s + me_{\bar{j}}}$ due to the combinatorial sector axiom. Therefore, the union of the $d$ topes of the form $T_{s + me_{\bar{j}}}$ exactly determines all possible neighbors of each vertex of $[\bar{d}]$ and hence the union of all topes/trees in question. The proposition now follows from the fact that $|T^{-1}_{s+me_{\bar{j}}}(\bar{j})| = s_{\bar{j}} + m$ for each $\bar{j}\in [\bar{d}]$.
\end{proof}

This property also implies the following well-known \emph{Spread-out simplices condition}:

\begin{corollary}[\cite{FlagArr}, Proposition 8.2(b)]
    Let $m$ be a positive integer, and $\tilde{P} = s + m\Delta^{d-1}$ where $s$ is a lattice point in $(n-m)\Delta^{d-1}$. In a fine mixed subdivision of $n\Delta^{d-1}$, there are at most $m$ cells inside $\tilde{P}$ that are translations of the unit simplex $\Delta^{d-1}$. (These cells are also called \emph{unmixed cells}.)
\end{corollary}

\begin{proof}
    When $d = 1$ this statement is trivial, so assume $d \geq 2$. It is clear that the unmixed cells' left degree vectors are (distinct) permutations of $d, 1, \dots, 1$. If there are at least $m+1$ unmixed cells in $\tilde{P}$, then the union of these trees has total left degree at least $(m+1)d+(n-m-1)$ (since there are at least $m+1$ vertices in $[n]$ with degree $d$ and the others have degree at least 1). However, \cref{prop:largetreeunion} implies that the total right degree of the union of these trees is at most $(n-m)+md$, which contradicts the fact that the total left degree and total right degree of a bipartite graph are equal.
\end{proof}

\subsection{Deletion and contraction}

Similar to matroids, we can construct sub-objects from fine mixed subdivisions, generic tropical oriented matroids, matching ensembles, extended tope arrangements, and (pre-)trianguloids. Here we provide one unifying definition that apply to all of the aforementioned objects.

\begin{definition}
    Let $\mathcal{G}$ be a collection of (distinct) subgraphs of $K_{n, d}$, optionally with a requirement that all vertices in $[n]$ and/or $[\bar{d}]$ have degree at least 1. For $i\in [n]$ and $\bar{j}\in \bar{d}$, the \emph{$i$-deletion} (resp. \emph{$\bar{j}$-contraction}) of $\mathcal{G}$, denoted $\mathcal{G}_{\setminus i}$ (resp. $\mathcal{G}_{/\bar{j}}$), is obtained via removing vertex $i$ (resp. $\bar{j}$) from each subgraph and then removing all duplicate graphs and graphs that no longer satisfy the minimum degree requirements for $\mathcal{G}$.
    
    For subsets $I\subseteq [n]$ and $\bar{J}\subseteq [\bar{d}]$, the \emph{$(I, \bar{J})$-minor} of $\mathcal{G}$, denoted $\mathcal{G}|_{I, \bar{J}}$, is obtained by deleting and contracting all elements in $[n]$ and $[\bar{d}]$ that are not in $I$ or $\bar{J}$, respectively. (It is clear that the deletions and contractions can be done in any order.)
\end{definition}

\begin{remark}
    In the language of fine mixed subdivisions of $n\Delta^{d-1}$, the $i$-deletion is the same as shrinking the $i$-th unit simplex to a single point and the $\bar{j}$-contraction is the same as taking the (induced) fine mixed subdivision of the facet opposite the vertex $n\cdot e_{\bar{j}}$ of $n\Delta^{d-1}$.
\end{remark}

\begin{proposition} All of the following objects are closed under taking minors: 
    \begin{enumerate}
        \item[(a)] \textnormal{(\cite{TropicalOM}, Proposition 4.7 and 4.8)} (Extended) (generic) tropical oriented matroids, and hence also fine mixed subdivisions and trianguloids;
        \item[(b)] Matching ensembles;
        \item[(c)] Extended tope arrangements;
        \item[(d)] Pre-trianguloids.
    \end{enumerate}
\end{proposition}

\begin{proof}
\begin{enumerate}
    \item[(b)] The closure axiom ensures that deleting edges from some graphs do not create two distinct partial matchings on the same set of vertices, so we still have exactly one partial matching for each pair of subsets. Since all of these partial matchings are already present in the original matching ensemble, the left/right linkage axioms are automatically preserved.

    \item[(c)] It is clear that if two graphs are compatible, then their subgraphs are also compatible. If after taking a minor there are two topes with the same position, then by \cref{lem:compat}(a) they are not compatible, so no two topes have the same position.
    
    Therefore, it suffices to show that after taking a minor there is at least one tope with a given position $v$. \begin{itemize}
        \item If we contract some index $\bar{j}\in [\bar{d}]$, the tope $T_{v'}$ where $v'_{\bar{j}} = 0$ and $v'_{\bar{j}'} = v_{\bar{j}'}$ for all $\bar{j}'\neq \bar{j}$ and will have position $v$ after contraction.
        \item If we delete some index $i\in [n]$, then take an arbitrary $\bar{j}\in [\bar{d}]$. If $T_{v + e_{\bar{j}}}(i) = \bar{j}$ then this tope will have position $v$ after deletion. Otherwise suppose that this coordinate is $\bar{j}'$. Since extended tope arrangements are pre-trianguloids, they satisfy the combinatorial sector condition of \cref{def:trianguloid}, and hence $T_{v + e_{\bar{j}'}} (i) = \bar{j}'$ and this tope will have position $v$ after deletion.
    \end{itemize}
    
    \item[(d)] We first show that any two ``adjacent'' topes of a pre-trianguloid whose positions $v$ and $v'$ differ by some $e_{\bar{j}} - e_{\bar{j}'}$ are compatible. By \cref{prop:treeunion}, since both positions are vertices of the same simplex $(v - e_{\bar{j}}) +\Delta^{d-1}$, these two topes are both subgraphs of a spanning tree. Since a tree cannot contain two distinct matchings on the same set of vertices (or else the union contains a cycle), any two subgraphs of it must be pairwise compatible.

    Contraction of a pre-trianguloid is simply taking the subset of topes whose right degrees are zero at the contracted index, so the same conditions are trivially satisfied. 
    
    If two topes have the same position vector after deleting an index, then they must be adjacent before deletion and hence must be compatible with each other (and hence the same after deletion). The existence of a tope with any given position follows from the same argument as in the proof of (c). 
    
    It remains to show that the combinatorial sector condition is preserved under deletion. Suppose we delete index $i\in [n]$. If $u, u'\in (n-1)\Delta^{d-1}$ are such that $u - e_{\bar{j}} = u' - e_{\bar{j}'}$ for some $\bar{j}, \bar{j}'\in [\bar{d}]$, then let $v = u + e_{\bar{j}'} = u' + e_{\bar{j}}$. If $T_v(i) = \bar{j}''$ (which may be the same as $\bar{j}$ or $\bar{j}'$), then the original condition implies that $T_{u + e_{\bar{j}''}}(i) = T_{u' + e_{\bar{j}''}}(i) = \bar{j}''$, and hence $T_{u + e_{\bar{j}''}}$ and $T_{u' + e_{\bar{j}''}}$ become $T_u$ and $T_{u'}$ respectively after deletion. Since the condition is satisfied with the first pair of topes, it is also satisfied for the second pair, as desired.

\end{enumerate}
\end{proof}

\begin{corollary}[Resolving \cite{FieldLattice}, Question 3.35]\label{cor:pretriext}
    The topes of a pre-trianguloid are pairwise compatible, so pre-trianguloids and extended tope arrangements are cryptomorphic objects.
\end{corollary}

\begin{proof}
    If any two topes of a pre-trianguloid are incompatible, then they disagree on a matching between $I\subseteq [n]$ and $\bar{J}\subseteq [\bar{d}]$. If we take the $(I, [\bar{d}])$-minor of the pre-trianguloid, the two topes will both have position $\mathbf{1}_{\bar{J}}$ in the new pre-trianguloid and remain distinct, which is impossible.

    Therefore, pre-trianguloids are extended tope arrangements, and combined with \cref{thm:extpretri} we see that the two are cryptomorphic.
\end{proof}

We also note that minors are also used to obtain matchings in a matching ensemble from a generic tropical oriented matroid:

\begin{proposition}[\cite{Ensembles}, Proposition 4.3 and Lemma 4.5]\label{prop:extraction}
    Given a generic tropical oriented matroid $\mathcal{T}$, for any $I\subseteq [n]$ and $\bar{J}\subseteq [\bar{d}]$ such that $|I| = |\bar{J}|$, let $M_{I, \bar{J}}$ be the unique tope of $\mathcal{T}|_{I, \bar{J}}$ with position $\mathbf{1}_{\bar{J}}$. The matchings $\{M_{I, \bar{J}}\}$ form a matching ensemble $\mathcal{M}$. We say that $\mathcal{M}$ is \emph{extracted} from $\mathcal{T}$.

    Moreover, matching ensembles extracted from distinct generic tropical oriented matroids are also distinct.
\end{proposition}

Note that the extracted matching ensemble is naturally contained in the extended generic tropical oriented matroid.

\section{From Right Linkage to Tree Compatibility}\label{sec:equivaxiom}

The primary goal of this section is to provide a proof to \cref{thm:TOMensembles}. Since pre-trianguloids and extended tope arrangements are also cryptomorphic to matching left-semi-ensembles, it suffices to show that the additional right-linkage axiom from matching ensembles translates to the condition that any two trees $\mathbb{T}(v)$ in \cref{prop:treeunion} are compatible. 

First, we mention one additional property of (extended) tope arrangements:

\begin{proposition}[\cite{FieldLattice}, Proposition 3.22 and Corollary 3.24]
    Given an extended tope arrangement $\mathcal{T}$ and a lattice point $w\in (n+1)\Delta^{d-1}$, let $\overline{\textnormal{supp}}(w)$ be the set of indices $\bar{j}\in [\bar{d}]$ for which $w_{\bar{j}}\geq 1$. The intersection of topes $\bigcap_{\bar{j}\in \overline{\textnormal{supp}}(w)} T_{w-e_{\bar{j}}}$ is a partial left semi-matching $\Omega(w)$ with right degree vector $w - \mathbf{1}_{\overline{\textnormal{supp}}(w)}$. In particular, the neighbors of $\bar{j}$ in $\Omega(w)$ are exactly the neighbors in $T_{w-e_{\bar{j}}}$.
    Moreover, we have $T_v = \bigcup_{\bar{j}\in [\bar{d}]} \Omega(v + e_{\bar{j}})$.
\end{proposition}

We call $\Omega(w)$ a \emph{Chow covector} of $\mathcal{T}$. These covectors originate from the Chow polytope of the variety of degenerate matrices (\cite{MaxMinors}, Section 5), and are closely related to the minimal transversals of the underlying (pointed) matching field (\cite{Fields}, Theorem 1).

It turns out that the \emph{union} of these topes is also interesting when right linkage is satisfied.

\begin{proposition}\label{prop:othertreeunion}
    Given a matching ensemble $\mathcal{M}$, let $\mathcal{T}$ be the extended tope arrangement obtained by iterated amalgamation from $\mathcal{M}$. For a lattice point $w\in (n+1)\Delta^{d-1}$, the union of topes $\bigcup_{\bar{j}\in \overline{\textnormal{supp}}(w)} T_{w-e_{\bar{j}}}$ is a spanning tree $\bar{\mathbb{T}}(w)$ on $[n]\sqcup \overline{\textnormal{supp}}(w)$ whose left degrees are all either 1 or 2. In particular, we have $T_{w-e_{\bar{j}}} = \Omega(w) \sqcup (\bar{\mathbb{T}}(w))^{\leftarrow \bar{j}}$.
\end{proposition}

\begin{proof}
    It suffices to prove this proposition for $\overline{\textnormal{supp}}(w) = [\bar{d}]$: if $\overline{\textnormal{supp}}(w)\neq [\bar{d}]$, then we can first take the $([n], \overline{\textnormal{supp}}(w))$-minor of $\mathcal{M}$ and $\mathcal{T}$. (Vertices not in $[n]\sqcup \overline{\textnormal{supp}}(w)$ will not be connected to any edges among all relevant topes. It is also clear that taking minor commutes with amalgamation.)

    In this case, we note that the Chow covector $\Omega(w)$ has exactly $n+1-d$ edges. Let $I$ be the set of vertices in $[n]$ with degree 1 in $\Omega(w)$ (so $|I| = n+1-d$). Since $\Omega(w)$ is a common subgraph of all $T_{w - e_{\bar{j}}}$, if we delete all indices in $I$ from $\mathcal{T}$, these topes will now have distinct right degree vectors $\mathbf{1}_{[\bar{d}]\setminus \{\bar{j}\}}$ and the same left degree vector $\mathbf{1}_{[n]\setminus I}$. In particular, these are now all partial matchings in the original matching ensemble. Using the right linkage axiom on these partial matchings, we see that the union is a spanning tree $\bar{\mathbb{T}}_{[n]\setminus I, [\bar{d}]}$ with left degree vector $\mathbf{2}_{[n]\setminus I}$. Therefore, the union of the original topes is $\Omega(w)\sqcup \bar{\mathbb{T}}_{[n]\setminus I, [\bar{d}]}$, which is a spanning tree on $[n]\sqcup [\bar{d}]$ with left degree vector $\mathbf{2}_{[n]\setminus I} + \mathbf{1}_I$. 
    
    Recovering $T_{w-e_{\bar{j}}}$ from $\bar{T}(w)$ can be shown in the same way as \cref{prop:ipush}, taking into account that the edges in $\Omega(w)$ are never in a $\bar{j}$-push but appear in all topes in question.
\end{proof}

In other words, while the left linkage axiom translates to the fact that the union of topes around a right-side-up unit simplex $u + \Delta^{d-1}$ is a tree, the additional right linkage axiom translates to the fact that the union around an \emph{upside-down} unit simplex $w + \nabla^{d-1}$ is also a tree, where $\nabla^{d-1}:= \textnormal{conv}(\{-e_{\bar{j}}\mid \bar{j}\in [\bar{d}]\})$. (This upside-down simplex might not be fully contained in $n\Delta^{d-1}$ in case $w$ does not have full support, in which case we only take the vertices that are in $n\Delta^{d-1}$.) We call these two types of trees \emph{$\Delta$-linkage covectors} and \emph{$\nabla$-linkage covectors}, respectively.

\begin{remark}
    We note that $\nabla$-linkage by itself also implies the combinatorial sector axiom. Indeed, suppose that $v - e_{\bar{j}} = v' - e_{\bar{j}'}$, then let $w = v + e_{\bar{j}'} = v' + e_{\bar{j}}$. Since there are no neighbors of $\bar{j}$ in $(\bar{\mathbb{T}}(w))^{\leftarrow \bar{j}}$ and one neighbor in $(\bar{\mathbb{T}}(w))^{\leftarrow \bar{j}'}$ for any $\bar{j}'\neq \bar{j}$, we see that $T^{-1}_{v}(\bar{j})$ is a superset of $T^{-1}_{v'}(\bar{j})$ (by one element). Therefore, $\nabla$-linkage is a strictly stronger condition than $\Delta$-linkage, and encompasses both left and right linkage from matching ensembles.
\end{remark}

In the language of fine mixed subdivisions, a $\Delta$-linkage covector encodes the unique Minkowski cell containing the corresponding right-side-up simplex, and a $\nabla$-linkage covector encodes the unique \emph{parallelotope} containing the corresponding upside-down simplex. See \cref{fig:linkagecovector} for an example. (Note that since a parallelotope is a zonotope, it corresponds to a graph whose left degrees are all either 1 or 2.) We leave the claim that a fine Minkowski cell contains an upside-down simplex if and only if it is a parallelotope as an exercise for the reader.

\begin{figure}[h!]
    \centering
    \includegraphics{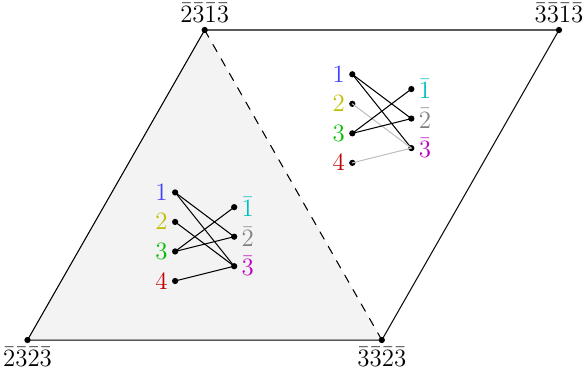}
    \caption{Example of a $\Delta$-linkage covector $\mathbb{T}((0,1,2))$ and a $\nabla$-linkage covector $\bar{\mathbb{T}}((1,1,3))$ from one of the fine Minkowski cells in the fine mixed subdivision of \cref{fig:FMSexample}. For the $\nabla$-linkage covector, the edges that belong to the Chow covector $\Omega((1,1,3))$ are in gray. Note that the two linkage covectors are the same, as they encode the same cell that contains both the corresponding right-side-up simplex $(0,1,2)+\Delta^2$ (lightly shaded) and the upside-down simplex $(1,1,3)+\nabla^2$ (unshaded).}\label{fig:linkagecovector}
\end{figure}

We now show that $\nabla$-linkage implies compatibility between $\Delta$-linkage covectors.

\begin{proposition}\label{prop:toperefine}
    Given an extended tope arrangement $\mathcal{T}$, for each lattice point $u\in (n-1)\Delta^{d-1}$, let $\mathbb{T}(u) = \bigcup_{\bar{j}\in [\bar{d}]} T_{u+e_{\bar{j}}}$. If for any lattice point $w \in (n+1)\Delta^{d-1}$, the graph $\bar{\mathbb{T}}(w) = \bigcup_{\bar{j}\in \overline{\textnormal{supp}}(w)} T_{w - e_{\bar{j}}}$ is acyclic, then any tope that is a subgraph of $\mathbb{T}(u)$ is a tope of $\mathcal{T}$.
\end{proposition}

\begin{proof}
    For this proof, we treat the tree $\mathbb{T} = \mathbb{T}(u)$ as a hypertree with vertex set on $[\bar{d}]$ (see \cref{rem:hypertree}). A tope $T$ that is a subgraph of $\mathbb{T}$ is the same as picking one of the vertices of each edge; we will call each tope an \emph{(edge) orientation} in this new formulation, where each edge is \emph{oriented to} the vertex we picked. For convenience, we will say that the edge $i$ is \emph{oriented towards} a vertex $\bar{j}$ (which is not necessarily a vertex of the edge) if it is in the same connected component of $\mathbb{T}\setminus i$ as the vertex we oriented the edge to. 
    
    We say that an orientation is \emph{good} if it is in $\mathcal{T}$. By \cref{prop:treeunion}, for each vertex $\bar{j}\in [\bar{d}]$, the $\bar{j}$-pull of $\mathbb{T}$ is good. In the language of hypertrees, this means that there is a good orientation where each edge is oriented towards $\bar{j}$.
    
    A \emph{path} $P$ between two vertices $\bar{j}$ and $\bar{j}'$ of the hypertree is the (unique) sequence of (distinct) edges $i_1, \dots, i_{\ell}$ that connects the two vertices, where any two adjacent edges $i_k$ and $i_{k+1}$ share exactly one vertex $\bar{j}_k$ (we also let $\bar{j}_0 = \bar{j}$ and $\bar{j}_{\ell} = \bar{j}'$ for convenience). Suppose we have two edge orientations $T$ and $T'$ in $\mathcal{T}$ that differ only on this path $P$, where the edges in $P$ are oriented towards $\bar{j}$ in $T$ and towards $\bar{j}'$ in $T'$. Let $v$ and $v'$ be the position of $T$ and $T'$ respectively; it is clear that we have $v + e_{\bar{j}'} = v' + e_{\bar{j}} = w$ for some $w\in (n+1)\Delta^{d-1}$. We see that $T\cup T'$ is a subgraph of $\bar{\mathbb{T}}(w)$. For each $k =0, 1, \dots, \ell $, let $T_k$ be the orientation obtained by orienting $i_1, \dots, i_{k-1}$ towards $\bar{j}$ and $i_{k}, \dots, i_{\ell}$ towards $\bar{j}'$ (i.e. the edges of $P$ are oriented ``away from $\bar{j}_k$''). Note that $T_0 = T'$ and $T_{\ell} = T$. It is clear that $T_k \subseteq T\cup T'$ as well, and has position $w - e_{\bar{j}_k}$. If the tope $T_{w - e_{\bar{j}_k}}$ of $\mathcal{T}$ at this position is not $T_k$, then it is clear that the union of these two topes contains a cycle, which implies that $\bar{\mathbb{T}}(w)$ also contains the same cycle. Therefore, we see that $T_1, \dots, T_{\ell - 1}$ are good. We call this process of obtaining more good orientations ``performing a \emph{tug-of-war} along $P$'' (since it essentially interpolates between a $\bar{j}$-pull and a $\bar{j}'$-pull when restricted to $P$).

    \begin{figure}[h!]
        \centering
        \includegraphics{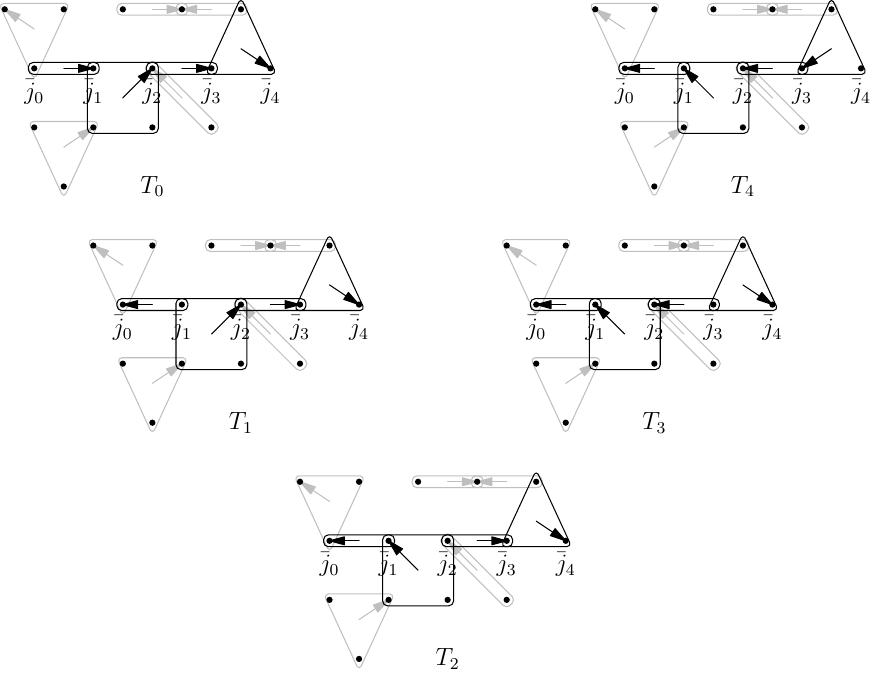}
        \caption{An example of a tug-of-war, where we generate three good orientations $T_1, T_2, T_3$ assuming that $T_0$ and $T_4$ are good. The orientation of each hyper-edge is indicated by an arrow. The gray hyper-edges are not on the path and their orientations are shared across all five orientations.}
    \end{figure}

    We show that we can use a series of tug-of-war to obtain any orientation of $\mathbb{T}$, starting from only the $\bar{j}$-pulls. Let us proceed by induction on the number $n$ of edges of $\mathbb{T}$, where the base case $n = 1$ is obvious (any orientation is a $\bar{j}$-pull in this case). We can also ignore any edge containing only a single vertex (since there is only one possible orientation for them). 
    
    Now suppose that we have shown the statement for all hypertrees with $n-1$ edges for $n\geq 2$. Consider a edge $i$ of $\mathbb{T}$ that only has one non-leaf vertex $\bar{j}_0$ (this edge exists since in the original bipartite graph formulation, removing all leaves in $[\bar{d}]$ necessarily gives a tree with at least two leaves in $[n]$). Let $\bar{J}$ be the set of vertices that are not leaves of $i$ (i.e. in the same connected component as $\bar{j}_0$ if we remove edge $i$). Consider an orientation $T$ of $\mathbb{T}$, of which there are two possible cases: \begin{itemize}
        \item Edge $i$ is oriented at vertex $\bar{j}_0$. We can remove $i$ from $T$ (and all leaf vertices of this edge) and $\mathbb{T}$ and obtain $T\setminus \{i\}$ in $\mathbb{T}\setminus \{i\}$ by induction hypothesis. Since the $\bar{j}$-pull of $\mathbb{T}$ has the same orientation for this edge as $T$ for all $\bar{j}\in \bar{J}$ and this edge is never in any path between two vertices of $\bar{J}$, the same series of tug-of-war that gives $T\setminus \{i\}$ in $\mathbb{T}\setminus \{i\}$ will give $T$ in $\mathbb{T}$.
        \item Edge $i$ is oriented at a leaf vertex $\bar{j}_1\notin \bar{J}$. For each $\bar{j}\in \bar{J}$, since any path between $\bar{j}_1$ and $\bar{j}$ must pass through vertex $\bar{j}_0$, we can perform a tug-of-war between the $\bar{j}_1$-pull and $\bar{j}$-pull of $\mathbb{T}$ (which only differ on the path between the two vertices) to get a ``modified $\bar{j}$-pull'', which differs from the $\bar{j}$-pull only in edge $i$ (which is now oriented at $\bar{j}_1$). Since all modified $\bar{j}$-pulls agree with $T$ on edge $i$, we can perform the same argument as the previous case to obtain $T$.
    \end{itemize}

    This completes the inductive step. Since all $\bar{j}$-pulls are good, we see that all orientations are good, as desired.
\end{proof}

\begin{proof}[Proof of \cref{thm:TOMensembles}]
    The statement that there is an injection from generic tropical oriented matroids to matching ensembles (via extraction) is \cref{prop:extraction}. For the reverse direction, we note that matching ensembles give rise to an extended tope arrangement and a collection of trees $\mathbb{T}(u)$, one for each $u\in (n-1)\Delta^{d-1}$, via the iterated amalgamation process. By \cref{prop:toperefine}, any tope that refines a tree is in the arrangement, so if two trees $\mathbb{T}$ and $\mathbb{T}'$ are incompatible, then there are two topes $T\subseteq \mathbb{T}$ and $T'\subseteq \mathbb{T}'$ that are also incompatible with each other. This means that the trees in the collection satisfy the compatibility axiom. Since there are exactly $\binom{n+d-2}{n-1}$ such trees, by \cref{rem:oneaxiom} we see that these trees encode a fine mixed subdivision and hence a generic tropical oriented matroid. 
    
    Since each partial matching of the original matching ensemble is a subgraph of some tope (as a result of the iterated amalgamation), it is clear that the extraction of the generic tropical oriented matroid is indeed the original matching ensemble. This shows that the extraction is also a surjection, so we have the desired bijection.
\end{proof}

\begin{remark}
    As a consequence of \cref{thm:TOMensembles}, we see that extended tope arrangements that satisfy the $\nabla$-linkage property in \cref{prop:othertreeunion} (which essentially arises from the right linkage axiom of the underlying matching ensemble) are cryptomorphic to pre-trianguloids that satisfy the hexagon axiom. 

    In fact, assuming that the conditions for extended tope arrangements (or pre-trianguloids) hold, right linkage and hexagon axiom are directly equivalent to each other: if we delete the Chow covector $\Omega(w)$ from each tope and exchange the role between $[n]$ and $[\bar{d}]$, then right linkage becomes strong linkage of \cref{thm:linkageequiv} while hexagon axiom becomes three-element linkage. This provides a shorter though less direct proof to \cref{thm:TOMensembles}.
\end{remark}

\section{A Linkage Covector for Trees}\label{sec:treelinkage}

Unlike the proof of \cref{thm:TOMensembles} in the previous section, the proof that trianguloids are cryptomorphic to fine mixed subdivisions (\cref{thm:trianguloids}) uses the tree linkage axiom of  instead of the compatibility axiom in \cref{thm:FMSaxioms} (see \cref{rem:oneaxiom}). In this section we use the tree linkage axiom to construct linkage covectors that each connects a ``local'' set of trees arising from a fine mixed subdivision. These linkage covectors have similar but more complex structures as the linkage covectors for (semi-)matchings, as they can no longer be defined using bipartite graphs.

\begin{definition}
    Given a fine mixed subdivision of $n\Delta^{d-1}$ defined by a collection of spanning trees $\mathbb{T}(u)$ (one for each $u\in (n-1)\Delta^{d-1}$) and a lattice point $t\in (n-2)\Delta^{d-1}$, define a \emph{tree-linkage graph} $\mathbb{G}(t)$ to be an edge-labeled graph on the vertex set $[\bar{d}]$, where we connect a pair of vertices $\bar{j}$ and $\bar{j}'$ if and only if $\mathbb{T}(t+e_{\bar{j}})$ and $\mathbb{T}(t+e_{\bar{j}'})$ differ in only one edge $(i, \bar{j})$ and $(i', \bar{j}')$. This edge is labeled on both ends like this: $\bar{j} \xleftrightarrow{i\quad i'}\bar{j}'$. For the vertex $\bar{j}$, we call $i$ the \emph{near label} for $\bar{j}$.
    
    We also use $G(t)$ to denote the same graph without edge labels.
\end{definition}

It is not difficult to see that $G(t)$ is also the ``facet-adjacency graph'' between the $d$ fine Minkowski cells whose positions are $t + e_{\bar{j}}$ for $\bar{j}\in [\bar{d}]$. 

\begin{remark}\label{rem:basicproperties}
    Note that the tree-linkage axiom translates to the property that for each non-leaf neighbor $i\in [n]$ of $\bar{j}$ in $\mathbb{T}(t+e_{\bar{j}})$, there is exactly one edge in $\mathbb{G}(t)$ incident to $\bar{j}$ for which the near label is $i$. (Note that since $RD(\mathbb{T}(t+e_{\bar{j}})) = t + e_{\bar{j}} + \mathbf{1}_{[\bar{d}]}$, $\bar{j}$ cannot be a leaf.) From \cref{lem:adjtrees}, the two endpoints of an edge cannot have the same label.
\end{remark}

\begin{proposition}\label{prop:treelinkagetree}
    For each lattice point $t\in (n-2)\Delta^{d-1}$, $G(t)$ is a spanning tree of $[\bar{d}]$.
\end{proposition}

\begin{proof}
    We first show that the graph is acyclic. Suppose that there is a cycle involving the vertex $\bar{j}_0$ like $\dots \bar{j}_\ell \xleftrightarrow{i_{\ell}\quad i'_0} \bar{j}_0 \xleftrightarrow{i_0\quad i'_{1}} \bar{j}_1 \dots$ in $\mathbb{G}(t)$. Denote $\mathbb{T}_k = \mathbb{T}(t+e_{\bar{j}_k})$ for convenience, then we can obtain $\mathbb{T}_1, \mathbb{T}_2, \dots, \mathbb{T}_\ell, \mathbb{T}_0 $ from $\mathbb{T}_0$ by repeatedly replacing one edge at a time. Note that the set of neighbors of $\bar{j}$ in this tree $\mathbb{T}_k$ is only modified at the first and last step, where in the first step we removed $i_0$ and in the last step we added $i_0'$. From \cref{rem:basicproperties} we see that $i_0 \neq i_0'$, which means that the final tree cannot be the same as the $\mathbb{T}_0$. Therefore, this cycle cannot exist.

    We now show that the graph is connected by showing that there is a path from $\bar{j}_0$ to $\bar{j}^*$ for any $\bar{j}_0, \bar{j}^*\in [\bar{d}]$. Suppose that in the tree $\mathbb{T}(t+e_{\bar{j}_0})$, the first vertex along the (unique) path from $\bar{j}_0$ to $\bar{j}^*$ is some $i_0\in [n]$. (In other words, $(i_0, \bar{j}_0)$ is an edge of $(\mathbb{T}(t+e_{\bar{j}_0}))^{\leftarrow\bar{j}^*}$.) This means that there is an edge of the form $\bar{j}_0 \xleftrightarrow{i_0\quad i'_1}\bar{j}_1$ in $\mathbb{G}(t)$. If $\bar{j}_1 = \bar{j}^*$ then we are done immediately. 
    Otherwise, note that from \cref{lem:adjtrees}, if $H_0 = \mathbb{T}(t+e_{\bar{j}_0})\setminus (i_0, \bar{j}_0)$, then $\bar{j}_0$ is in the same connected component $H_0^{(1)}$ of $H$ as $i'_1$, and $i_0, \bar{j}_1, \bar{j}^*$ are in the opposite connected component $H_0^{(2)}$. 
    
    Therefore, after adding the edge $(i'_1, \bar{j}_1)$, $i'_1$ is further away from $\bar{j}^*$ than $\bar{j}_1$ in the new tree $\mathbb{T}(t+e_{\bar{j}_1})$, and hence there exists a different $i_1$ adjacent to $\bar{j}_1$ that is closer to $\bar{j}^*$. If we define $H_1^{(1)}$ and $H_1^{(2)}$ as the two connected components of $H_1 = \mathbb{T}(t+e_{\bar{j}_1})\setminus (i_1, \bar{j}_1)$ where the former component contains $\bar{j}_1$, then the vertex set of $H_0^{(1)}$ is a strict subset of that of $H_1^{(1)}$, since the previous edge replacement did not affect connectivity of $H_0^{(1)}$, and this component is now also connected to $\bar{j}_1$ (and possibly some additional vertices). We also have that $\bar{j}^*\in H_1^{(2)}$ by the same argument as with $H_0$.
    
    Hence we can always repeat this process (which gives a path along the graph $\mathbb{G}(t)$), never repeating any $\bar{j}_k$'s (as $\mathbb{G}(t)$ is acyclic), until we reach $\bar{j}_\ell = \bar{j}^*$ for some $\ell$.
\end{proof}

\begin{remark}
    A proof of connectivity of $G(t)$ can also be found in \cite{TriangulationTOM}, Lemma 4.10, using machinery of totally unimodular matrices.
\end{remark}

Similar to semi-matchings and linkage covectors, there is an analogous way to (partially) obtain trees from tree-linkage graphs.

\begin{definition}
    Given a tree-linkage graph $\mathbb{G}(t)$ and $\bar{j}\in [\bar{d}]$, let the \emph{$\bar{j}$-pull} of $\mathbb{G}(t)$ be a subgraph $G\subseteq K_{n,d}$ consisting of the following edges: for each edge of $\mathbb{G}(t)$, if $\xrightarrow{i'}\bar{j}'$ is the end closer to $\bar{j}$, then add the edge $(i', \bar{j}')$ to $G$. We use $\mathbb{G}(t)^{\rightarrow\bar{j}}$ to denote the resulting graph.
\end{definition}

\begin{proposition}\label{prop:pullplustope}
    For a given lattice point $t\in (n-2)\Delta^{d-1}$, let $T_t = \bigcap_{\bar{j}\in [\bar{d}]} \mathbb{T}(t+e_{\bar{j}})$. We have the following: \begin{enumerate}
        \item[(a)] For each $\bar{j}\in[\bar{d}]$, $\mathbb{T}(t+e_{\bar{j}}) = T_t\sqcup \mathbb{G}(t)^{\rightarrow\bar{j}}$.
        \item[(b)] $T_t$ is a tope, i.e. $LD(T_t) = \mathbf{1}_{[n]}$.
    \end{enumerate}
\end{proposition}

\begin{proof}
\begin{enumerate}
    \item[(a)] If $\bar{j}$ and $\bar{j}'$ are adjacent in $\mathbb{G}(t)$ via an edge $\bar{j} \xleftrightarrow{i\quad i'}\bar{j}'$, then the $\bar{j}$-pull and $\bar{j}'$-pull differ by exactly one edge $(i, \bar{j})$ and $(i', \bar{j}')$, just like $\mathbb{T}(t+e_{\bar{j}})$ and $\mathbb{T}(t+e_{\bar{j}}')$. If there is an edge $(i^*, \bar{j}^*) \in \mathbb{G}(t)^{\rightarrow\bar{j}} \setminus \mathbb{T}(t+e_{\bar{j}})$, then it is also in $\mathbb{G}(t)^{\rightarrow\bar{j}'} \setminus \mathbb{T}(t+e_{\bar{j}'})$. This argument applies transitively along $\mathbb{G}(t)$ to show that $(i^*, \bar{j}^*)\in \mathbb{G}(t)^{\rightarrow\bar{j}^*} \setminus \mathbb{T}(t+e_{\bar{j}^*})$, which is impossible due to \cref{rem:basicproperties}. Therefore, we see that $\mathbb{G}(t)^{\rightarrow\bar{j}} \subseteq \mathbb{T}(t+e_{\bar{j}})$ for all $\bar{j}$.

    Now let $T^{(\bar{j})} = \mathbb{T}(t+e_{\bar{j}}) \setminus \mathbb{G}(t)^{\rightarrow\bar{j}}$ for each $\bar{j}$. From the same argument as above we see that $T^{(\bar{j})} = T^{(\bar{j}')}$ whenever $\bar{j}$ and $\bar{j}'$ are adjacent in $\mathbb{G}(t)$, and is hence the same set for all $\bar{j}\in [\bar{d}]$, which we denote by $T'_t$. Therefore $T_t$ contains $T'_t$. On the other hand, any edge in $T_t$ cannot be replaced in any tree, and hence will not appear in any $\bar{j}$-pull, so it is also contained in $T'_t$. This means that $T_t = T'_t$, as desired.

    \item[(b)] Since $\mathbb{G}(t)$ is a spanning tree on $[\bar{d}]$, $\mathbb{G}(t)^{\rightarrow\bar{j}}$ has $d-1$ edges and hence $T_t$ has exactly $n$ edges, so it suffices to show that no vertices in $[n]$ has degree greater than $1$ in $T_t$. Indeed, if $i$ is connected to two vertices $\bar{j}$ and $\bar{j}'$ in all trees, then the edge $(i, \bar{j})$ is replaceable in $\mathbb{T}(t+e_{\bar{j}})$ and cannot be in $T_t$.

\end{enumerate}
\end{proof}

\begin{corollary}
    In a fine mixed subdivision of $n\Delta^{d-1}$ and a lattice point $t\in (n-2)\Delta^{d-1}$, the intersection of the $d$ cells whose positions are $t + e_{\bar{j}}$ for some $\bar{j}\in [\bar{d}]$ is exactly a single lattice point.
\end{corollary}

\begin{figure}[h!]
    \centering
    \includegraphics{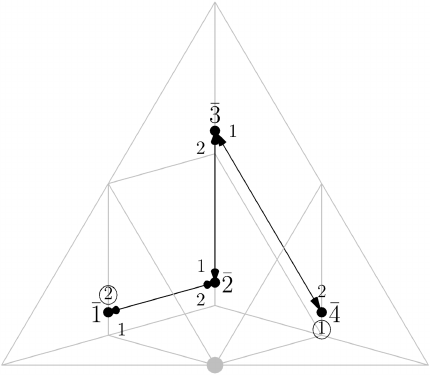}
    \caption{The only tree-linkage covector $\tilde{\mathbb{G}}(\mathbf{0})$ from the fine mixed subdivision of $2\Delta^3$ in \cref{fig:2d3} (reproduced in gray), drawn with vertices coinciding with the position of the corresponding cell. The vertex labels are marked with a circle. The lattice point shared by all four cells is marked with a large gray dot. (Note that this lattice point might not be contained in $t + 2\Delta^{d-1}$ in general.)}\label{fig:2d3linkagegraph}
\end{figure}

\begin{definition}
    Given a fine mixed subdivision and a lattice point $t\in (n-2)\Delta^{d-1}$, define a \emph{tree-linkage covector} $\tilde{\mathbb{G}}(t)$ to be the tree-linkage graph $\mathbb{G}(t)$ where each vertex is additionally labeled by some (zero or more) elements of $[n]$: let $T_t = \bigcap_{\bar{j}\in [\bar{d}]} \mathbb{T}(t+e_{\bar{j}})$, then vertex $\bar{j}$ is labeled with $i$ if and only if $(i, \bar{j})$ is an edge of $T_t$.

    For a vertex $\bar{j}\in [\bar{d}]$, define its \emph{augmented degree} to be the number of incident edges plus the number of vertex labels. Also, we define the \emph{$\bar{j}$-pull} $\tilde{\mathbb{G}}(t)^{\rightarrow \bar{j}} = T_t \sqcup \mathbb{G}(t)^{\rightarrow \bar{j}}$.
\end{definition}

By \cref{prop:pullplustope}, each $i\in [n]$ appears as a vertex label of a tree-linkage covector exactly once, and each tree $\mathbb{T}(t+e_{\bar{j}})$ is exactly recovered by the $\bar{j}$-pull of the covector. 

\begin{proposition}\label{prop:extendeddegree}
    For a lattice point $t\in (n-2)\Delta^{d-1}$, the augmented degree vector of $\tilde{\mathbb{G}}(t))$ is equal to $t + \mathbf{2}_{\bar{d}}$.
\end{proposition}

\begin{proof}
    Given a vertex of tree-linkage covector $\tilde{\mathbb{G}}(t)$, notice that the near label on each incident edge and each vertex label represent an edge that is in at least one of the spanning trees. Therefore, the augmented degree vector is the same as the right degree vector of the union $\bigcup_{\bar{j}\in[\bar{d}]} \mathbb{T}(t + e_{\bar{j}})$. The result then follows from \cref{prop:largetreeunion} where we take $m = 2$.
\end{proof}

We conclude this section by some additional properties of the labels on $\tilde{\mathbb{G}}(t)$ that may be of interest:

\begin{proposition}
    If there is a path in $\tilde{\mathbb{G}}(t)$ of the form \[\bar{j}_0 \xleftrightarrow{i_0\quad i'_1}\bar{j}_1 \xleftrightarrow{i_1\quad i'_2}\bar{j}_2 \dots \bar{j}_{\ell-1} \xleftrightarrow{i_{\ell-1}\quad i'_\ell}\bar{j}_\ell,\] then the following holds: \begin{enumerate}
        \item[(a)] Let $k$ be an integer between $0$ and $\ell-1$ inclusive. In the graph $H_k = \mathbb{T}(t+e_{\bar{j}_k})\setminus (i_k, \bar{j}_k)$, $\bar{j}_0, \dots, \bar{j}_k$ are in a different connected component from $\bar{j}_{k+1}, \dots, \bar{j}_{\ell}$.
        \item[(b)] For any $1\leq k\leq k'\leq \ell - 1$, $i'_{k}\neq i_{k'}$. (In other words, no identical edge labels can point at each other.)
        \item[(c)] If vertex $\bar{j}_k$ is labeled with $i$, then $i_{k'}\neq i$ for all $k' \geq k$ and $i'_{k'}\neq i$ for all $k' \leq k$. (In other words, no edge label can point at an identical vertex label.)
    \end{enumerate}
\end{proposition}

\begin{proof}
    (a) follows from the proof of connectivity for \cref{prop:treelinkagetree}. For (b), it suffices to show that $i'_1\neq i_{\ell - 1}$. The cases $\ell = 1$ and $\ell = 2$ are obvious from \cref{rem:basicproperties}, so we assume $\ell \geq 3$. From \cref{prop:pullplustope}(a) we see that $(i'_1, \bar{j}_1)$ and $(i_{\ell-1}, \bar{j}_{\ell-1})$ are both in $\tilde{\mathbb{G}}(t)^{\rightarrow \bar{j}_1} = \mathbb{T}(t+e_{\bar{j}_1})$, and from (a) we also have that $\bar{j}_1$ and $\bar{j}_{\ell-1}$ are in different connected components of $H_1 = \mathbb{T}(t+e_{\bar{j}_1})\setminus (i_1, \bar{j}_1)$, which means that $i'_1\neq i_{\ell-1}$. (c) follows from the same argument as (b).
\end{proof}

In other words, if we look at only the edges labels that are equal to $i$, they must all be oriented away from the unique vertex that is labeled $i$.

\section{The Big Picture}\label{sec:bigpicture}

In \cref{tab:summary} we list our current understanding of the axiomatics of generic tropical oriented matroids. There are two main classes of cryptomorphic objects: ones that only satisfy a property in the first column under ``Properties/Axioms'' (corresponding to pointed linkage matching fields), and ones that satisfy one or more properties that together span both columns (corresponding to generic tropical oriented matroids).

\begin{table}[h!]
\bgroup
\def\arraystretch{1.2}

\begin{tabular}{|cccc|}
\hline
\multicolumn{1}{|c|}{\textbf{Object}} &
  \multicolumn{1}{c|}{\textbf{Elements}} &
  \multicolumn{2}{c|}{\textbf{Properties/Axioms}} \\ \hline
\multicolumn{1}{|c|}{Pointed matching field} &
  \multicolumn{1}{c|}{\begin{tabular}[c]{@{}c@{}}Partial left semi-matchings\\ (Also right semi-matchings)\end{tabular}} &
  \multicolumn{1}{c|}{Linkage} &
   \\ \hline
\rowcolor[HTML]{C0C0C0} 
\multicolumn{1}{|r}{\cellcolor[HTML]{C0C0C0}Reduction $\downarrow$} &
  \multicolumn{1}{l}{\cellcolor[HTML]{C0C0C0}$\uparrow$ Completion} &
  \multicolumn{2}{c|}{\cellcolor[HTML]{C0C0C0}} \\ \hline
\multicolumn{1}{|c|}{Matching stack/ensemble} &
  \multicolumn{1}{c|}{Partial matchings} &
  \multicolumn{1}{c|}{\begin{tabular}[c]{@{}c@{}}Extended left linkage\\ (Closure + Left linkage)\end{tabular}} &
  \begin{tabular}[c]{@{}c@{}}Right\\ linkage\end{tabular} \\ \hline
\rowcolor[HTML]{C0C0C0} 
\multicolumn{1}{|r}{\cellcolor[HTML]{C0C0C0}Iterated amalgamation $\downarrow$} &
  \multicolumn{1}{l}{\cellcolor[HTML]{C0C0C0}$\uparrow$ Sub-matchings} &
  \multicolumn{2}{c|}{\cellcolor[HTML]{C0C0C0}} \\ \hline
\multicolumn{1}{|c|}{Extended tope arrangement} &
  \multicolumn{1}{c|}{} &
  \multicolumn{1}{c|}{Tope compatibility} &
   \\ \cline{1-1} \cline{3-4} 
\multicolumn{1}{|c|}{(Pre-)Trianguloid} &
  \multicolumn{1}{c|}{} &
  \multicolumn{1}{c|}{Combinatorial sector} &
  Hexagon \\ \cline{1-1} \cline{3-4} 
\multicolumn{1}{|c|}{} &
  \multicolumn{1}{c|}{} &
  \multicolumn{1}{c|}{Chow covector} &
   \\ \cline{3-4} 
\multicolumn{1}{|c|}{} &
  \multicolumn{1}{c|}{} &
  \multicolumn{1}{c|}{$\Delta$-linkage} &
   \\ \cline{3-4} 
\multicolumn{1}{|c|}{\multirow{-3}{*}{}} &
  \multicolumn{1}{c|}{\multirow{-5}{*}{\begin{tabular}[c]{@{}c@{}}Left semi-matchings\\ (Topes)\end{tabular}}} &
  \multicolumn{2}{c|}{$\nabla$-linkage} \\ \hline
\rowcolor[HTML]{C0C0C0} 
\multicolumn{1}{|r}{\cellcolor[HTML]{C0C0C0}$\Delta$-linkage covectors $\downarrow$} &
  \multicolumn{1}{l}{\cellcolor[HTML]{C0C0C0}$\uparrow$ Pulling} &
  \multicolumn{2}{c|}{\cellcolor[HTML]{C0C0C0}} \\ \hline
\multicolumn{1}{|c|}{} &
  \multicolumn{1}{c|}{} &
  \multicolumn{2}{c|}{Tree compatibility} \\ \cline{3-4} 
\multicolumn{1}{|c|}{\multirow{-2}{*}{\begin{tabular}[c]{@{}c@{}}Fine mixed subdivision\\ (or triangulation)\end{tabular}}} &
  \multicolumn{1}{c|}{\multirow{-2}{*}{Spanning trees}} &
  \multicolumn{2}{c|}{Tree linkage} \\ \hline
\rowcolor[HTML]{C0C0C0} 
\multicolumn{1}{|r}{\cellcolor[HTML]{C0C0C0}Subgraphs $\downarrow$} &
  \multicolumn{1}{l}{\cellcolor[HTML]{C0C0C0}$\uparrow$ Maximal graphs} &
  \multicolumn{2}{c|}{\cellcolor[HTML]{C0C0C0}} \\ \hline
\multicolumn{1}{|c|}{\begin{tabular}[c]{@{}c@{}}(Extended) Generic \\ tropical oriented matroid\end{tabular}} &
  \multicolumn{1}{c|}{\begin{tabular}[c]{@{}c@{}}Acyclic subgraphs\\ (Generated by elimination)\end{tabular}} &
  \multicolumn{2}{c|}{Type compatibility} \\ \hline
\end{tabular}

\egroup
\bigskip

\caption{A summary of the relationships between different combinatorial objects mentioned in this paper and their properties. Properties that span the same set of columns are equivalent to each other.}\label{tab:summary}
\end{table}

With the exception of generic tropical oriented matroids, one important commonality of the objects in this paper is that they require there to be exactly one graph for each possible left/right degree vector (or each pair of left and right degree vectors, in the case of matching stacks). This condition is absent for generic tropical oriented matroids, but is instead generated from the boundary axiom and the elimination axiom. It remains to be seen whether there is any direct connection between the elimination axiom and some formulation of the linkage axiom (such as the one suggestively named ``Elimination" in \cref{thm:linkageequiv}).

\subsection{Future directions}

There are several potential ways to generalize the work of this paper: \begin{itemize}
    \item Instead of considering fine mixed subdivisions of $P = n\Delta^{d-1}$, one can replace each summand of the Minkowski sum with a face of $\Delta^{d-1}$. The resulting polytope will be a \emph{generalized permutohedron}, whose fine mixed subdivisions biject to triangulations of a \emph{root polytope} (\cite{Permutohedra}, Section 12). This generalization translates to considering acyclic subgraphs of a particular subgraph $H\subseteq K_{n,d}$. We note that the theory of trianguloids is already generalized to this case (\cite{Trianguloids}, Section 5), and a generalization of tropical oriented matroids has been conjectured in \cite{TriangulationTOM}.
    \item The argument in \cref{sec:equivaxiom} suggests that there may be some sort of symmetry between $\Delta$-linkage and $\nabla$-linkage originating from left and right linkage, even though in the $P = n\Delta^{d-1}$ case the latter implies the former. This motivates considering fine mixed subdivisions of $P = n\Delta^{d-1} + m\nabla^{d-1}$ where this symmetry might be more apparent. This Minkowski sum is also related to the Tutte polynomial of a (poly)matroid (\cite{RootTutte, TutteLattice, UniversalTutte}).
    \item It might also be possible to generalize to \emph{non-fine} mixed subdivisions, where the corresponding graphs no longer need to be acyclic. 
\end{itemize}

\bigskip

\paragraph{\bfseries Acknowledgments.} 
The author would like to thank Alex Postnikov for introducing the author to the problem and their continued support. The author would also like to thank Georg Loho for many motivating discussions and useful references.

\bibliographystyle{alpha}
\bibliography{refs}

\end{document}